\documentclass{journal} % Aptara syntax

\usepackage[ruled]{algorithm2e}
\usepackage{xfrac}
\usepackage{wrapfig}

\SetAlFnt{\small}
\SetAlCapFnt{\small}
\SetAlCapNameFnt{\small}
\SetAlCapHSkip{0pt}
\IncMargin{-\parindent}
\usepackage{amsmath}
\usepackage{tensor}
\usepackage{float}
\usepackage{eurosym}
\usepackage{amssymb}
\usepackage{bm}
\newcommand{\etal}{\textit{et~al.}\xspace}
\usepackage{amsthm}
\newtheorem{definition}{Definition}
\newtheorem{theorem}{Theorem}
\newtheorem{corollary}{Corollary}
\newtheorem{lemma}{Lemma}
\newtheorem{claim}{Claim}
\title{On the Degree Distribution of  P\'{o}lya Urn Graph Processes}
\usepackage{authblk}
\author{Rasul Tutunov, Haitham Bou Ammar, Ali Jadbabaie, and Eric Eaton\\
\large {University of Pennsylvania, Philadelphia PA 19104}
}
\usepackage{fancyhdr}
\begin{document}

% Page heads
% NOTE! Affiliations placed here should be for the institution where the
%       BULK of the research was done. If the author has gone to a new
%       institution, before publication, the (above) affiliation should NOT be changed.
%       The authors 'current' address may be given in the "Author's addresses:" block (below).
%       So for example, Mr. Abdelzaher, the bulk of the research was done at UIUC, and he is
%       currently affiliated with NASA.

\begin{abstract}
This paper presents a tighter bound on the degree distribution of arbitrary P\'{o}lya urn  graph processes, proving that the proportion of vertices with degree $d$ obeys a power-law distribution $P(d) \propto d^{-\gamma}$ for $d \leq n^{\sfrac{1}{6}-\epsilon}$ for any $\epsilon > 0$, where $n$ represents the number of vertices in the network. Previous work by Bollob\'{a}s \etal formalized the well-known preferential attachment model of Barab\'{a}si and Albert, and showed that the power-law distribution held for $d \leq n^{\sfrac{1}{15}}$ with $\gamma = 3$.  Our revised bound represents a significant improvement over existing models of degree distribution in scale-free networks, where its tightness is restricted by the Azuma-Hoeffding concentration inequality for martingales. We achieve this tighter bound through a careful analysis of the first set of vertices in the network generation process, and show that the newly acquired is at the edge of exhausting Bollob\'as model in the sense that the degree expectation breaks down for other powers. 
\end{abstract}

%\category{C.2.2}{Computer-Communication Networks}{Network Protocols}

%\terms{Design, Algorithms, Performance}


%\acmformat{Gang Zhou, Yafeng Wu, Ting Yan, Tian He, Chengdu Huang, John A. Stankovic,
%and Tarek F. Abdelzaher, 2010. A multifrequency MAC specially
%designed for  wireless sensor network applications.}
% At a minimum you need to supply the author names, year and a title.
% IMPORTANT:
% Full first names whenever they are known, surname last, followed by a period.
% In the case of two authors, 'and' is placed between them.
% In the case of three or more authors, the serial comma is used, that is, all author names
% except the last one but including the penultimate author's name are followed by a comma,
% and then 'and' is placed before the final author's name.
% If only first and middle initials are known, then each initial
% is followed by a period and they are separated by a space.
% The remaining information (journal title, volume, article number, date, etc.) is 'auto-generated'.

%\begin{bottomstuff}
%\end{bottomstuff}

\maketitle

\section{Introduction}
The power-law degree distribution is an interesting property exhibited by many complex networks. Aiming at a better understanding of such a characteristic, Barab\'{a}si and Albert proposed a linear preferential attachment model for generating scale-free networks~\cite{BA1}. The definition of their process, however, was rather informal, as noted by Durret~\cite{Durret}. Since then, different precise forms of such a formulation have been studied in literature \cite{One,Two,Four}. Out of these, the Bollob\'{a}s \etal model \cite{Bollobas1}, adopted in this paper, stands out as it has been at the core of various developments in network studies \cite{BollobasRiodan,Three,Riodan,Sym}. Further, Bollob\'{a}s \etal's approach is quite general, as opposed to alternative interpretations~\cite{Hou2008}, in the sense that the degree distribution is size dependent.

%For these reasons, we adopt Bollob\'{a}s et.~al's formalization as the general framework for our analysis with the goal of tightening the degree distribution bounds. 

%Real-world networks can be represented as random graphs.  
%Namely, the proportion of vertices $P(d)$ with a degree $d$ follows a power-law distribution: $P(d) \propto d^{-\gamma}$  with $\gamma = 2.9 \pm 0.1$. 
%Aiming at a better understanding of such a distribution, Barab\'{a}si and Albert suggested a linear preferential attachment model for the generation of scale-free networks~\cite{BA1}.  The definition of their process, however, was rather informal, as noted by Durret~\cite{Durret}. Since then, different precise forms of the Barab\'{a}si-Albert (BA) model have been studied \cite{One,Two,Four}. Out of these, the Bollob\'{a}s et.~al~\cite{Bollobas1} model (detailed in Section~\ref{Sec:Bollo}) stands out as it has been at the core of various developments in network studies \cite{BollobasRiodan,Sym, Riodan,Three}. Further, Bollob\'{a}s et.~al's approach is general, as opposed to alternative interpretations~\cite{Hou2008}, in the sense that the degree distribution is size dependent. 
%For these reasons, we adopt Bollob\'{a}s et.~al's formalization as the general framework for our analysis with the goal of tightening the degree distribution bounds. 

In their approach, Bollob\'{a}s \etal introduce $n$-pairings (i.e., linear chord diagram models) as the procedure to formalize the preferential attachment process. Interestingly, this method allows for the definition of random graphs over $n$ vertices non-recursively.  The model also allows for both loops and multiple edges. The authors prove that the power-law degree distribution can be acquired with $\gamma=3$. The proof, however, operates only when the studied degree $d \leq n^{\sfrac{1}{15}}$. Such loose bounds restrict real-world applications of Bollob\'{a}'s theorem~\cite{Bollobas1} due to the need for extremely large networks (e.g., $10^{15}$ for an individual degree of $10$). 

Aiming at a more realistic setting, this paper contributes by first tightening the degree distribution bound above from $d \leq n^{\sfrac{1}{15}}$ to $d \leq n^{\sfrac{1}{14}-\epsilon}$ for any $\epsilon > 0$. Although successful, we show that this result can be further constricted. We introduce a careful analysis of the first set of vertices in the network generation process and show that the bound can be further improved to $d \leq n^{\sfrac{1}{6}-\epsilon}$ for any $\epsilon > 0$.  To our knowledge, this is the tightest bound for the degree distribution of scale-free networks in such a general setting discovered so far.  Finally, we present a corollary of the second theorem demonstrating that the newly acquired bound is tight to changes in degree exponents.

Our final results provide an improvement over previous work, showing that the degree distribution of vertices in arbitrary networks following from the preferential attachment process obey a power-law distribution $P(d) \propto d^{-\gamma}$ for $d \leq n^{\sfrac{1}{6}-\epsilon}$ for any $\epsilon > 0$. To illustrate, the set of inequalities attained for determining the exponents of the number of vertices are plotted for both theorems proved in this paper. In Figure~\ref{fig:test1}, we show the extend to which Bollob\'{a}'s method can be stretched without the careful analysis of the first set of vertices, leading to an exponent of $\sfrac{1}{14}$. After our careful analysis of the first set of vertices, we are capable of further extending that feasibility region to acquire an exponent of $\sfrac{1}{6}$, see Figure~\ref{fig:test2}.  
\begin{figure}
\centering
\begin{minipage}{.5\textwidth}
  \centering
  \includegraphics[scale=0.2]{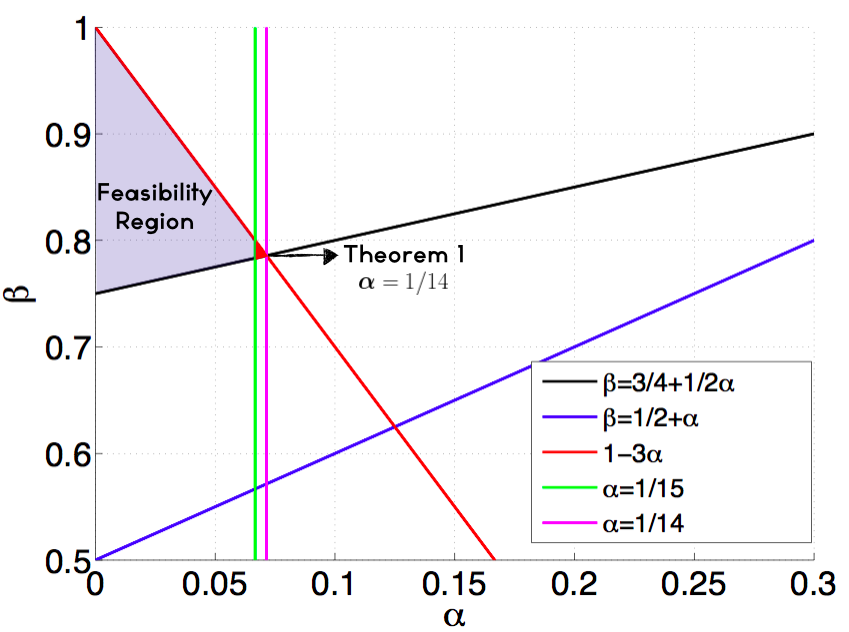}
  \caption{An illustration showing the result attained in Theorem 1, where Bollob\'{a}'s bound is tightened from $\sfrac{1}{15}$ to $\sfrac{1}{14}$.}
  \label{fig:test1}
\end{minipage}%
\begin{minipage}{.5\textwidth}
  \centering
  \includegraphics[scale=.2]{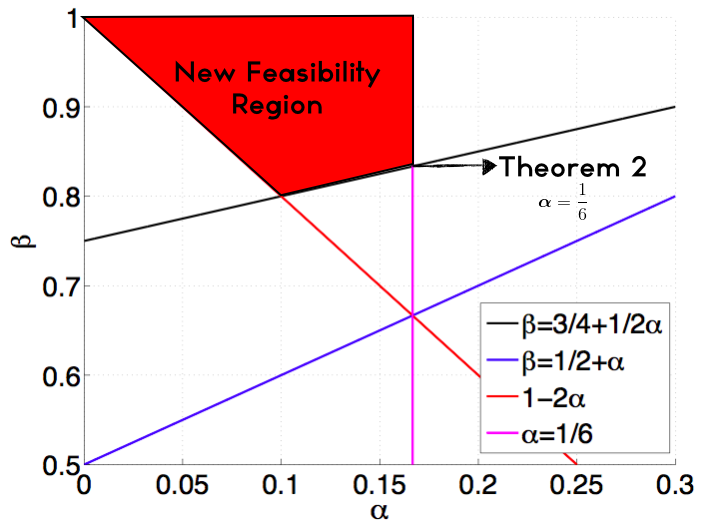}
  \caption{Illustrating the extension of the feasibility region due to our analysis in Theorem 2, being at the edge of exhausting Bollob\'{a}'s method.}
  \label{fig:test2}
\end{minipage}
\end{figure}

This newly acquired bound can be considered tight as its tightness is restricted by the Azuma-Hoeffding concentration inequality for martingales. Our approach is also at the edge of exhausting Bollob\'as \etal's approach, in the sense that the degree expectation breaks down for other powers.

\section{Background}\label{Sec:Back}
%This section presents background material needed for the remainder of the paper. Firstly, the Barab\'{a}si-Albert model used for the generation of scale-free networks is detailed. Secondly, the idea of $n$-pairing essential for the proofs is explained. Finally, a set of Lemmas (proved elsewhere) are then detailed. 

%\subsection{Networks}
%A network is represented by a graph $\mathbb{G}=\left(\mathbb{V},\mathbb{W}\right)$ consisting of a non-empty set of nodes (or vertices) $\mathbb{V}=\{v_{1},\dots,v_{N}\}$ and an $N\times N$ adjacency matrix $\mathbb{W}=[w_{ij}]$ where the elements of $w_{ij}$ indicate the connection from $v_{i}$ to $v_{j}$. 

\subsection{Preferential Attachment and the P\'{o}lya Urn Process}
The preferential attachment process introduced in~\cite{BA1} can be formalized as a combination of P\'{o}lya urn processes~\cite{Urn}, in the sense that each newly arriving connection can represent a new ball added to the urn corresponding to that vertex. For such a formalization, consider a P\'{o}lya urn process with a two urn model, where the number of balls in one urn represent the degree of a node $\bm{v}_{k}$, and those in the second denote the sum of the degrees for $\bm{v}_{1},\dots, \bm{v}_{k-1}$. The process starts at $n=k$, where node $\bm{v}_{k}$ has exactly $m>0$ connections to $\bm{v}_{1},\dots,\bm{v}_{k-1}$. Recognizing that at this stage the first and second urns start with $m$ and $\left(2k-3\right)m$ balls, respectively, it is easy to see that the evolution is a P\'{o}lya urn with strengths $\bm{\alpha}_{k}$ and $1-\bm{\alpha}_{k}$ with $\bm{\alpha}_{k}\sim\bm{\beta}\left(m,\left(2k-3\right)m\right)$, with $\bm{\beta}(\cdot)$ representing the Beta distribution. 

The aforementioned process enables an accurate definition of the preferential attachment model by setting $\bm{\alpha}_{1}=1$, and for $2\leq k \leq n$, $\bm{\alpha}_{k} \sim \bm{\beta}\left(m,\left(2k-3\right)m\right)$. Further, by letting $\bm{l}_{k}=\sum_{j=1}^{k}\bm{\Phi}_{k}$, with $\bm{\Phi}_{k}=\bm{\alpha}_{k}\prod_{j=k+1}^{n}\left(1-\bm{\alpha}_{j}\right)$ for $1 \leq k \leq n$, an edge between $\bm{v}_{k}$ and $\bm{v}_{j}$ is drawn if for some $i \in [1,m]$, we have $j=\bm{\kappa}\left(\bm{U}_{i,k}\bm{l}_{k-1}\right)$, with $\bm{\kappa}\left(a\right)= \min\left\{k:\bm{l}_{k} \geq a\right\}$ for $a\in[0,1]$ and $\left\{\bm{U}_{i,k}\right\}$ being a sequence of independent random variables for $1\leq i \leq m$ and $1 \leq k \leq n$.

Though appealing, Bollob\'{a}s \etal~\cite{Bollobas1}, among others, presented a procedure to formalize such preferential attachment models based on the concept of $n$-pairings~\cite{Bollobas3,Bollobas2} enabling easier analysis. This framework exhibits multiple advantages, such as constructing graphs  over $n$ vertices non-recursively, and providing compact and tractable representations of degree distributions. Furthermore, this method has been at the core of multiple alternative formalizations of preferential attachment \cite{BollobasRiodan,Three,Riodan,Sym} and therefore we adopt it in this paper as a general framework in the study of the degree distribution of scale-free networks. 

%It is worth noting that although the Barab\'{a}si-Albert model is a well known model for generating scale-free graphs, theoretical studies aiming at formalizing the power-law distribution received less attention. In this work, not only we provide formal theoretical analysis for the emergence of the power-law, but also generalize current results by providing the tightest achievable bounds under the adopted model.   

\subsection{$n$-Pairings and Graph Generation}\label{Sec:Pairing}
The idea of $n$-parings~\cite{Bollobas2,Bollobas3} is one of the essential steps required to generate graphs in the formalization introduced by Bollob\'{a}s \etal~\cite{Bollobas1}. An \emph{n-pairing} is a partition of the set $\{1,2,\dots,2n\}$ into pairs, so there are $\left(2n-1\right)!!=\sfrac{(2n)!}{\left(n!2^{n}\right)}$ $n$-pairings~\cite{Bollobas1}. These objects can be viewed as linearized chord diagrams (LCDs)~\cite{Bollobas2}. An LCD with $n$ chords consists of $2n$ distinct points on the x-axis paired by chords. %as shown in Figure~\ref{Fig:LCD}. 

%Two LCDs are considered the same when one can be turned into the other by moving the points on the x-axis without changing their order. Thinking of pairing as LCDs, we shall talk of chords and their left and right endpoints. 
Starting from $2n$-nodes, this process first generates a random matching between pairs of nodes.  
A directed graph $\bm{\Psi}\left(\mathcal{P}\right)$ is then formed from an $n$-pairing $\mathcal{P}$ as follows: starting from the left, merge all endpoints including the first right endpoint to form $v_{1}$. At this stage, merge all further endpoints up to the next right endpoint to form $v_{2}$. This procedure is repeated until it reaches $v_{n}$. Edges are created by replacing each pair by a directed edge from the vertex corresponding to its right endpoint to that corresponding to its left endpoint. As noted by Bollob\'{a}s \etal~\cite{Bollobas1}, if $\mathcal{P}$ is chosen uniformly at random from all  $\sfrac{(2n)!}{\left(n!2^{n}\right)}$ $n$-pairings, then $\bm{\Psi}\left(\mathcal{P}\right)$ has the same distribution as a random graph. Such a process can be used to formalize the widely known preferential attachment model of Barab\'{a}si and Albert, as detailed next.  %further formalized 

%$\mathbb{G}_{1}^{(n)}\in\mathfrak{G}_{1}^{(n)}$. 
%In this work, $n$-pairings are used to generate directed graphs and provide an alternative for describing its distribution. 

\subsection{Generating Graphs}\label{Sec:Bollo}
Bollob\'{a}s introduced two processes to formalize the scale-free degree distribution inherent to real-world networks. In the \emph{one-connection preferential attachment process} (Section~\ref{Sec:M1}) the goal is for a new node to make just one connection to the existing nodes in the graph. The \emph{multiple-connections preferential attachment process} (Section~\ref{Sec:M2}), generalizes this idea to $m>1$ connections. In this setting, the graphs are created by applying the one-connection process multiple times.  

%It is worth noting, that such processes have been introduced in~\cite{Bollobos1,Bollobos2}. %The authors, however, provide a rather intuitive argument on the equivalence between these and the Barab\'{a}si-Albert model~\cite{Bollobos2}. Here, we start our contributions by proving an equivalence between these processes and the Barab\'{a}si-Albert model. %providing a set of theorems (accompanied with their proofs) on such an equivalence. 

%In~\cite{Bollobos}, however, the authors provide a rather intuitive argument on the equivalence between these and the Barab\'{a}si-Albert model. 

\subsubsection{One-Connection Preferential Attachment Process}\label{Sec:M1}
For $m=1$, the goal is for each new node to make only one connection to those nodes that already exist in the graph. This process, $\left(\mathbb{G}_{1}^{(n)}\right)_{n\geq 0}$, is defined inductively so that $\mathbb{G}_{1}^{(n)}$ is a directed graph on $\{v_{i}: 1\leq i\leq n\}$. Formally:
\begin{definition}[One-Connection Preferential Attachment Process~\protect\cite{Bollobas1}:]
Start with $\mathbb{G}_{1}^{(1)}$, the graph with one vertex and one loop. Construct $\mathbb{G}_{1}^{(n)}$ from $\mathbb{G}_{1}^{(n-1)}$ by adding a vertex (node) $v_{n}$ with a single directed edge from $v_{n}$ to $v_{i}$, with $i$ chosen according to: 
 \begin{displaymath}
\mathbb{P}\mathrm{r}(i=s) = \left\{
     \begin{array}{ll}
       \frac{d^{\left(\mathbb{G}_{1}^{(n-1)}\right)}[v_{s}]}{2n-1} &:  1 \leq s \leq n-1\\
       \sfrac{1}{2n-1} &: s=t
     \end{array}
   \right.
   \enspace .
\end{displaymath} 
\end{definition}

%\begin{claim}[Equivalence $m=1$ Setting]
%One connection preferential attachment process, $\left(\mathbb{G}^{(n)}_{1}\right)_{n\geq 0}$ is equivalent to the Barab\'{a}si-Albert model with $m=1$.
%\end{claim}
%\begin{proof}
%\comH{Proof Sketch Comes here.}
%\end{proof}

\subsubsection{Multiple-Connections Preferential Attachment Process}\label{Sec:M2}
For the case of $m>1$ (i.e., multiple connections), $m$ edges from the new node are added one at a time. The process $\left(\mathbb{G}^{(n)}_{m}\right)_{n\geq 0}$ is defined as following: 
\begin{definition}[Multiple-Connections Preferential Attachment Process~\protect\cite{Bollobas1}:]
The process $\left(\mathbb{G}^{(n)}_{m}\right)_{n\geq 0}$ is defined by running $\left(\mathbb{G}^{(n)}_{1}\right)_{n\geq 0}$ on a set of vertices $v^{\prime}_{1},v^{\prime}_{2}\dots$ to create $\mathbb{G}^{(mn)}_{1}$. The graph $\mathbb{G}^{(n)}_{m}$ is then formed by identifying the vertices $v_{1}^{\prime},\dots, v^{\prime}_{m}$ to form $v_{1}$, $v^{\prime}_{m+1},\dots,v^{\prime}_{2m}$ to form $v_{2}$, and so forth. 
\end{definition}

%\begin{claim}[Equivalence $m>1$ Setting]
%The multiple connection preferential attachment process $\left(\mathbb{G}_{m}^{(n)}\right)_{n\geq 0}$ is equivalent to the Barab\'{a}si-Albert model with $m>1$. 
%\end{claim}

%\begin{proof}
%\comH{Proof Comes Here!}
%\end{proof}

The probability space for directed graphs on $n$ vertices will be denoted by $\mathcal{G}_{m}^{(n)}$, where $\mathbb{G}_{m}^{(n)} \in \mathcal{G}_{m}^{(n)}$ has the distribution derived from the multiple-connection preferential attachment process described above. As noted by Bollob\'{a}s and Riordan~\cite{Bollobas2}, such a distribution has an alternative description in terms of $n$-pairings with the advantage of a simple and non-recursive definition of the distribution of $\mathbb{G}_{m}^{(n)}$. Due to space constraints, the details of such derivations are omitted in this paper. Interested readers are referred to either the supplementary material accompanying this paper or to Bollob\'{a}s \etal\cite{Bollobas1,Bollobas2} for a more thorough explanation.

\section{Revised Bounds on the Degree Distribution}\label{Sec:MainResults}
This section presents the main results of this paper in the form of two theorems. Due to space constraints, we only provide proof sketches here; extended proofs can be found in the supplementary material accompanying this paper. Theorem~\ref{Theo:Them} tightens the bounds on the in-degree distribution of scale-free networks from $n^{\sfrac{1}{15}}$ to $n^{\sfrac{1}{14}-\epsilon}$ for any $\epsilon > 0$.  Our second main result, summarized in Theorem~\ref{Theo:US}, further tightens the above bound to attain $n^{\sfrac{1}{6}-\epsilon}$ for any $\epsilon > 0$. This bound is the tightest discovered for such a general setting so far. We achieve $n^{\sfrac{1}{6}-\epsilon}$ by performing a detailed and careful analysis of the degree distribution of the first set of vertices in scale-free networks. %These results are summarized in the statement of Theorem~\ref{Theo:US}. 
%carefully analyzing the degree distribution of the first set of vertices to acquire an even tighter bound of $1<d<n^{\sfrac{1}{10}-\epsilon}$ for some $\epsilon >0$, Theorem~\cite{Theo:US}. 

%\subsection{Attaining }
%In this section we tighten the bound attained in~\cite{Bollobas} from $n^{\sfrac{1}{15}}$ to $n^{\sfrac{1}{14}}$. This new bound is captured by the following theorem: 
\begin{theorem}
\label{Theo:Them}
Let $m\geq 1$ be fixed and $\left(\mathbb{G}_{m}^{(n)}\right)_{n\geq 0}$ be the process defined in Section~\ref{Sec:M2}. For any $\epsilon >0$ and $1\leq d \leq n^{\sfrac{1}{14}-\epsilon}$: 
\begin{equation*}
\mathbb{P}\mathrm{r}\left[\frac{N\left(n,m,d\right)}{n}\sim \frac{1}{d^{3}}\right] \rightarrow 1, \ \text{as $n\rightarrow \infty$} \enspace ,
\end{equation*}
with $N\left(n,m,d\right)$ being the number of vertices in graph $\mathbb{G}_{m}^{(n)}$ with an in-degree $d$. 
\end{theorem}
Intuitively, the above theorem states that the probability for vertices to follow a power-law distribution tends to 1 as the $n$ grows large. The studied degree, in this case, should be restricted in the range of $1\leq d\leq n^{\sfrac{1}{14}-\epsilon}$ for some $\epsilon > 0$, which is tighter than that derived by Bollob\'{a}s \etal\cite{Bollobas1} (i.e., $1\leq d\leq n^{\sfrac{1}{15}}$).  Note that later, we further tighten this bound to $n^{\sfrac{1}{6}-\epsilon}$ as summarized by Theorem~\ref{Theo:US}.

%\begin{proof}
%Our proof is closely related to that of Bollob\'{a}s et.~al~\cite{Bollobas1} with major differences including the tightness of the derived bounds. At a high level, the proof is performed in two steps. In the first step, we consider the $m=1$ scenario and then generalize to $m>1$. Considering $m = 1$, the proof is established by proving three Lemmas. 

%First, we will prove the following: 
%\begin{lemma}\label{lemma_1}
%Let $D_k$ be the sum of the total degrees of nodes $v_1, v_2, \ldots v_k$ in graph $\mathbb{G}^{(n)}_{1}$. Then,
%\begin{equation*}
%\mathbb{P}\mathrm{r}\left[ |D_k - 2\sqrt{4kn}| \ge 4\sqrt{n\log n} \right] \approx \mathcal{O}(n^{-\sfrac{7}{4}}) \enspace .
%\end{equation*}
%\end{lemma}
\begin{proof}
Our proof is closely related to that of Bollob\'{a}s \etal\cite{Bollobas1} with major differences including the tightness of the derived bounds. At a high level, the proof is performed in two steps. In the first step, we consider the $m=1$ scenario and then generalize to $m>1$. Considering $m = 1$, the proof is established by proving three Lemmas. 

First, we will prove the following: 
\begin{lemma}\label{lemma_1}
Let $D_k$ be the sum of the total degrees of nodes $v_1, v_2, \ldots v_k$ in graph $\mathbb{G}^{(n)}_{1}$. Then,
\begin{equation*}
\mathbb{P}\mathrm{r}\left[ |D_k - 2\sqrt{4kn}| \ge 4\sqrt{n\log n} \right] \approx \mathcal{O}(n^{-\sfrac{7}{4}}) \enspace .
\end{equation*}
\end{lemma}

\begin{proof}
Fixing $1 \le s \le n - k$, we bound the probability of the event $\bm{A} = \{D_k = 2k + s \}$.
The event $\bm{A}$ is equivalent to the event that a set of new nodes $\left(\text{i.e.,}\{v_{k+1}, v_{k+2}, \ldots,  v_n\}\right)$ makes exactly $s$ links with the collection $\{v_1, v_2, \ldots,  v_k\}$. This is true since in the $\mathbb{G}^{(n)}_{1}$  process each new node creates exactly one outgoing link with previous nodes. Precisely, $\{v_1, v_2, \ldots, v_k\}$ exhibit exactly $k$ edges among themselves.

To acquire the value of $\mathbb{P}\mathrm{r}[\bm{A}]$, we will use a well-know result from the $n$-pairings theory~\cite{Bollobas2}. Namely, if $n$-pairings are chosen uniformly randomly, then the corresponding directed graph will exhibit the same distribution as $\mathbb{G}^{(n)}_{1}$:
%We will use the result of [\textbf{make reference on connection between pairing }] which states that if $n-$pairing is chosen uniformly random then the corresponding directed graph for that $n-$ pairing will have the same distribution as $\mathbb{G}^{(n)}_{1}$, i.e: 
\begin{equation}\label{form_1}
\mathbb{P}\mathrm{r}\left[\bm{A}\right] = \frac{N(s)}{(2n-1)!!} \enspace ,
\end{equation}
with $N(s)$ representing the number of $n$-pairings in which the $k^{th}$ right-end point corresponds to the LCD-node $2k+s$, with exactly $s$ LCD-nodes in the collection $C_{\text{left}} \triangleq \{1,2,\ldots 2k + s\}$ being left end-points for some $s$ nodes in the collection $C_{\text{right}} \triangleq \{ 2k + s + 1, 2k + s + 2, \ldots 2n \}$. Furthermore, the total number of all $n$-pairings is given by $(2n-1)!! = \sfrac{(2n)!}{2^nn!}$. Notice that the number of ways to pair LCD-node $2k+s$ with: (1) one of the elements in $C_{\text{left}}$, and (2) exactly $2k - s - 2$ elements of $C_{\text{left}}$ with themselves is given by $(2k + s - 1)\left(\begin{array}{c}
2k+s-2 \\
s
\end{array}\right)(2k-2)!!$. Similarly, the number of ways to pair exactly $(2n-2k - 2s)$ numbers of $C_{\text{right}}$ among each other is given by $\left(\begin{array}{c}
2n - 2k - s \\
s
\end{array}\right)(2n-2k - 2s)!!$.

Having these derivations, we can write:
\begin{equation*}
N(s) = s!(2k + s - 1)\left(\begin{array}{c}
2k+s-2 \\
s
\end{array}\right)(2k-2)!!\left(\begin{array}{c}
2n - 2k - s \\
s
\end{array}\right)(2n-2k - 2s)!! 
\enspace,
\end{equation*} 
with $s!$ representing all different ways in which $s$ $C_{\text{left}}$ LCD-nodes can be paired with $s$ nodes in $C_{\text{right}}$.

Given this result, we can rewrite Equation~\ref{form_1} as:
\begin{equation}\label{formula_for_Dk}
\mathbb{P}\mathrm{r}\left[\bm{A}\right] = \frac{(2k + s - 1)!(2n - 2k - s)!n!2^{s+1}}{s!(k-1)!(n-k-s)!(2n)!} \enspace .
\end{equation} 
Note that for a fixed $k$, the function $f(s) = \sfrac{\mathbb{P}\mathrm{r}\left[D_k = 2k + s + 1 \right]}{\mathbb{P}\mathrm{r}\left[D_k = 2k + s\right]} = \sfrac{2(2k+s)(n-k-s)}{(s+1)(2n-2k-s)}$ decreases with $s$. Considering the case where $\sfrac{2(2k+s)(n-k-s)}{(s+1)(2n-2k-s)} = 1$, we have the following solutions for $s$:
\begin{align*}
&s_{01} =  \lceil -2k + \sqrt{4kn - 2n + \sfrac{1}{4}} + \sfrac{1}{2}\rceil   \enspace,\\\nonumber 
&s_{02} = \lfloor -2k - \sqrt{4kn - 2n + \sfrac{1}{4}} - \sfrac{1}{2}\rfloor  \enspace. 
\end{align*}

Using the above statements, we can prove:

\begin{claim}
For all $s^{\prime}$ such that $s^{\prime} <n,$ and $s^{\prime} \ne s_{01},s_{02}$:
\begin{equation*}
\mathbb{P}\mathrm{r}\left[D_k = 2k + s^{\prime}\right] \le \mathbb{P}\mathrm{r}\left[D_k = 2k + s_{01}\right] \enspace .
\end{equation*}
\end{claim}

\begin{proof}
The opposition case can be stated as follows: there $s^{\prime}\ne s_{01}, s_{02}$ such that
$\mathbb{P}\mathrm{r}\left[D_k = 2k + s^{\prime} \right] > \mathbb{P}\mathrm{r}\left[D_k = 2k + s_{01}\right] $. Two cases are to be considered: 
\begin{enumerate}
\item[a)] \textbf{Case One:} $s^{\prime} < s_{01}$ then $f\left(s^{\prime}\right) \ge f(s_{01})$, therefore $f(s^{\prime}) \geq f(s^{\prime} + 1) \geq \ldots \geq f(s_{01} - 1) \geq f(s_{01})$, i.e: $\mathbb{P}\mathrm{r}\left[D_k = 2k + s'\right] \leq \mathbb{P}\mathrm{r}\left[D_k = 2k + s' + 1\right] \leq \mathbb{P}\mathrm{r}\left[D_k = 2k + s' + 2\right] \leq \ldots \leq \mathbb{P}\mathrm{r}\left[D_k = 2k + s_{01}\right]$. Therefore $\mathbb{P}\mathrm{r}\left[D_k = 2k + s'\right] \leq \mathbb{P}\mathrm{r}\left[D_k = 2k + s_{01}\right]$, which contradicts the statement of $s^{\prime}$.

\item[b)] \textbf{Case Two:} $s^{\prime} > s_{01}$ then $f(s') \le f(s_{01})$ then $f(s') \leq f(s_{01}) = 1$, therefore $f(s') \leq f(s' - 1) \leq \ldots \leq f(s_{01} + 1) \leq f(s_{01}) = 1$, i.e: $\mathbb{P}\mathrm{r}\left[D_k = 2k + s'\right] \leq \mathbb{P}\mathrm{r}\left[D_k = 2k + s' - 1\right] \leq \mathbb{P}\mathrm{r}\left[D_k = 2k + s' - 2\right] \leq \ldots \leq \mathbb{P}\mathrm{r}\left[D_k = 2k + s_{01}\right]$. Therefore, $\mathbb{P}\mathrm{r}\left[D_k = 2k + s^{\prime}\right] \leq \mathbb{P}\mathrm{r}\left[D_k = 2k + s_{01}\right]$, which contradicts the statement of $s^{\prime}$. 
\end{enumerate}
\end{proof}

\begin{claim}
Let $l$ be any positive integer, then for large $n$
\begin{equation}\label{ineq_21}
\mathbb{P}\mathrm{r}\left[D_k = 2k + s_{01} + l\right]\le e^{\left(-\sfrac{l(l-1)}{4n} \right)} \enspace .
\end{equation}
\end{claim}
\begin{proof}
Consider the ratio:
\begin{align}\label{ratio_1}
&\frac{f(s+1)}{f(s)} = \frac{(2k+s + 1)(n-k-s -1)(s + 1)(2n - 2k - s)}{(s+2)(2n - 2k -s - 1)(2k + s)(n - k - s)} = \\\nonumber
&\left[ 1 - \frac{2k+1}{(s + 2)(2k + s)} \right]\left[1 - \frac{n - k}{(n - k - s)(2n - 2k - s - 1)} \right] \le  \\\nonumber
&\left[1 - \frac{2k -1}{2n^2} \right]\left[1 - \frac{n-k}{2n^2}\right] \le e^{\left(-\sfrac{2k-1}{2n^2} - \sfrac{n-k}{2n^2}\right)} \le e^{\left(-\sfrac{1}{2n}\right)} \enspace .
\end{align}
Therefore, $f(s_{01}+1)\le e^{\left(-\sfrac{1}{2n}\right)}$. That is., $\mathbb{P}\mathrm{r}\left[D_k = 2k + s_{01} + 1\right] \le e^{\left(-\sfrac{1}{2n}\right)}\mathbb{P}\mathrm{r}\left[D_k = 2k + s_{01}\right]$. This result can be further generalized: $\frac{f(s_{01}+1)}{f(s_{01})} \le e^{\left(-\sfrac{1}{2n}\right)};\frac{f(s_{01}+2)}{f(s_{01}+1)} \le e^{\left(-\sfrac{1}{2n}\right)}; \ldots ; \frac{f(s_{01}+l)}{f(s_{01} + l - 1)} \le e^{\left(-\sfrac{1}{2n}\right)}$. Therefore, the multiplication of these inequalities gives:
\begin{equation*}
\frac{f(s_{01} + l)}{f(s_{01})} \le e^{\left(-\sfrac{l}{2n} \right)} \enspace ,
\end{equation*}
and because $f(s_{01})\le 1$, then $f(s_{01} + l) \le e^{\left(-\sfrac{l}{2n} \right)}$, i.e $\frac{\mathbb{P}\mathrm{r}\left[D_k = 2k + s_{01} + l\right]}{\mathbb{P}\mathrm{r}\left[D_k = 2k + s_{01} + l - 1\right]}\le e^{\left(-\sfrac{l}{2n} \right)}$. 

Given the collection of inequalities:
\begin{equation*}
\begin{cases} \mathbb{P}\mathrm{r}\left[D_k = 2k + s_{01} + l\right] \le e^{\left(-\sfrac{l}{2n} \right)}\mathbb{P}\mathrm{r}\left[D_k = 2k + s_{01} + l - 1\right]  \\ \mathbb{P}\mathrm{r}\left[D_k = 2k + s_{01} + l - 1\right] \le e^{\left(-\sfrac{l-1}{2n} \right)}\mathbb{P}\mathrm{r}\left[D_k = 2k + s_{01} + l - 2\right] \\\vdots \\\mathbb{P}\mathrm{r}\left[D_k = 2k + s_{01} + 1\right] \le e^{\left(-\sfrac{1}{2n} \right)}\mathbb{P}\mathrm{r}\left[D_k = 2k + s_{01}\right] \end{cases}
\end{equation*}
leads us to: 
\begin{equation*}
\mathbb{P}\mathrm{r}\left[D_k = 2k + s_{01} + l\right] \le e^{\left(-\frac{l + (l-1) + \ldots + 1}{2n}\right)} = e^{\left(-\sfrac{l(l-1)}{4n}\right)} \enspace .
\end{equation*} 
\end{proof}

Similarly, it can be shown that for any positive integer $l$:
\begin{equation}\label{ineq_22}
\mathbb{P}\mathrm{r}\left[D_k = 2k + s_{01} - l\right]\le e^{\left(-\sfrac{l(l-1)}{4n} \right)} \enspace .
\end{equation}

\begin{claim}\label{int_arg}
\end{claim}
For any $1 \le k \le n$:
\begin{equation}
\mathbb{P}\mathrm{r}\left[ |D_k - 2k - s_{01}| \ge 3\sqrt{n\log n} \right] \approx \mathcal{O}(n^{-\sfrac{7}{4}}) \enspace .
\end{equation}

\begin{proof}
Using the inequalities of Equations~\ref{ineq_21} and~\ref{ineq_22}, it can be seen that:
\begin{align}\label{sum_1}
&\mathbb{P}\mathrm{r}\left[ |D_k - 2k - s_{01}| \ge 3\sqrt{n\log n} \right] = \mathbb{P}\mathrm{r}\left[ D_k - 2k - s_{01} \ge 3\sqrt{n\log n} \right] + \\\nonumber
& \mathbb{P}\mathrm{r}\left[ |D_k - 2k - s_{01}| \le -3\sqrt{n\log n} \right] = \sum_{l = \lceil3\sqrt{n\log n}\rceil}^{\infty}\mathbb{P}\mathrm{r}\left[D_k = 2k + s_{01} + l\right] + \\\nonumber
&\sum_{l = \lfloor3\sqrt{n\log n}\rfloor}^{\infty}\mathbb{P}\mathrm{r}\left[D_k =  2k + s_{01} - l\right] \le 2\sum_{l= \lceil3\sqrt{n\log n}\rceil}^{\infty}e^{\left(-\sfrac{l(l-1)}{4n}\right)} \enspace .
\end{align}
Let $l_0 = \lceil3\sqrt{n\log n}\rceil$ and consider the ratio:
\begin{align*}
\frac{e^{\left(-\sfrac{(l+1)l}{4n}\right)}}{e^{\left(-\sfrac{l(l-1)}{4n}\right)}} = e^{\left(-\sfrac{l}{2n}\right)} \le  e^{\left(-\sfrac{l_0}{2n}\right)} \enspace .
\end{align*}
Therefore, the tight bound for sum in Equation~\ref{sum_1} has the following form:
\begin{align}\label{sum_2}
\sum_{l= \lceil3\sqrt{n\log n}\rceil}^{\infty}e^{\left(-\sfrac{l(l-1)}{4n}\right)} \le e^{\left(-\sfrac{l_0(l_0 - 1)}{4n}\right)}\frac{1}{1 - e^{\left(-\sfrac{l_0}{4n}\right)}} \approx \mathcal{O}\left(n^{-\sfrac{7}{4}}\right) \enspace .
\end{align}
Combining the inequalities in Equations~\ref{sum_1} and~\ref{sum_2}, we arrive at: 
\begin{equation*}
\mathbb{P}\mathrm{r}\left[ |D_k - 2k - s_{01}| \ge 3\sqrt{n\log n} \right] \approx \mathcal{O}(n^{-\sfrac{7}{4}})\nonumber \enspace .
\end{equation*}
\end{proof}

\begin{claim}\label{sub_arg}
Let event $\bm{B} = \{|D_k - 2k - s_{01}| \ge 3\sqrt{n\log n}\}$ and event $\bm{C} = \{|D_k - 2\sqrt{kn}|\ge 4\sqrt{n\log n}\}$. Then for large $n$:
\begin{equation}
\mathbb{P}\mathrm{r}\left[ \bm{C} \right] \le \mathbb{P}\mathrm{r}\left[ \bm{B} \right] \enspace .
\end{equation}
\end{claim}

\begin{proof}
We first prove that for large $n$:  $\bm{C}\subseteq \bm{B}$. This is true since for each $k$, such that $1 \le k\le n$:
\begin{equation}\label{ineqs01}
|s_{01} + 2k - 2\sqrt{kn}| \le 2\sqrt{n} \enspace .
\end{equation}
Therefore: 
%using (\ref{ineqs01}):
\begin{align}\label{ineqs02}
|D_k  - 2k - s_{01}| &= |D_k - 2k - s_{01} + 2\sqrt{kn} - 2\sqrt{kn}| \\\nonumber
&\ge  |D_k - 2\sqrt{kn}| - | s_{01} +2k - 2\sqrt{kn}| \ge |D_k - 2\sqrt{kn}| - 2\sqrt{n} 
\end{align}
Finally, using (\ref{ineqs02}) for large $n$ we have:
\begin{align}
&\bm{C} = \{|D_k - 2\sqrt{kn}|\ge 4\sqrt{n\log n}\} \subseteq \{|D_k - 2\sqrt{kn}|\ge 2\sqrt{n} + 3\sqrt{n\log n}\} \subseteq \\\nonumber
&\{|D_k - 2k + s_{01}|\ge 3\sqrt{n\log n}\} = \bm{B} 
\end{align}
Therefore, $\mathbb{P}\mathrm{r}\left[ \bm{C} \right] \le \mathbb{P}\mathrm{r}\left[ \bm{B} \right]$.
\end{proof}

Using claims \ref{int_arg} and \ref{sub_arg}, we can conclude that $\mathbb{P}\mathrm{r}\left[ |D_k - 2\sqrt{4kn}| \ge 4\sqrt{n\log n} \right] \approx \mathcal{O}\left(n^{-\sfrac{7}{4}}\right)$ thus proving  Lemma~\ref{lemma_1}. 
\end{proof}

The next claim deals with computing the conditional degree distribution of node $v_{k+1}\in\mathbb{G}^{(n)}_{1}$ given the sum of node (i.e., $\{v_1, v_2, \ldots, v_k\}$) degrees:

\begin{claim}
Let $d_{k+1}$ be the total degree of node $v_{k+1}$ in $\mathbb{G}^{(n)}_{1}$, and $D_k$ be the sum of total degrees of nodes $v_1, \ldots v_k$, then:
\begin{equation}\label{cond_prob_expression}
\mathbb{P}\mathrm{r}\left[ d_{k+1} = d + 1| D_k = 2k + s \right] = \frac{2^d(s + d)!(n - k - s)!(2n - 2k - s - d - 1)!}{(n - k - s - d)!(2n - 2k - s)!}
\enspace ,
\end{equation}
where $0 \le s \le n - k$.
\end{claim}

\begin{proof}
To yield $\mathbb{P}\mathrm{r}\left[ d_{k+1} = d + 1| D_k = 2k + s \right]$, we will compute the ratio of number of $n$-pairings defining events $\{ d_{k+1} = d + 1\}$ and $\{D_k = 2k + s \}$ to the number of $n$-pairings defining that of $\{D_k = 2k + s\}$.

Consider the number of $n$-pairings of LCD-nodes $\{C_{\text{left}} = \{1,2,\ldots 2k + s\}$ and $C_{\text{right}} = \{2k + s + 1 ,2k + s + 2, \ldots 2n\}$. For each n pairing of $C_{\text{left}}$ that defines event $\{D_k = 2k + s\}$ there are exactly $s!\left(\begin{array}{c}
2n - 2k - s \\
s
\end{array}\right)(2n - 2k - 2s - 1)$!!    
different $n-$pairings of $C_{\text{right}}$ which define the event $\{D_k = 2k + s\}$. Notice that event $\{d_{k+1} = d + 1\}$ is true if and only if LCD-node $2k + s + d + 1$ is the right end-point for some other LCD-node, and LCD-nodes $\{2k + s + 1, 2k+s+2, \ldots, 2k + s + d\}$ are left points for $d$ LCD-nodes from $C_{\text{shifted right}} \triangleq \{2k + s + d + 1, 2k + s + d + 2, \ldots, 2n\}$. 

Therefore, the number of $n$-pairings of the set $C_{\text{right}}$ which define event $\{d_{k+1} = d + 1\}$ is:
\begin{equation}\label{cond_prob_comb}
(s + d)!\left(\begin{array}{c}
2n - 2k - s - d - 1 \\
s + d - 1
\end{array}\right)(2n - 2k - 2s - 2d - 1)!!
\enspace .
\end{equation}
Furthermore, the number of ways to generate LCD-node $2k + s + d - 1$ is $s + d$. Also, the number of ways to pair $s + d - 1$ LCD-nodes from the set $\{1,2,\ldots 2k + s + d\}$ with these from $C_{\text{shifted right}}$ is: 
$\left(\begin{array}{c}
2n - 2k - s - d - 1 \\
s + d - 1
\end{array}\right)$, each allowing for $(s + d - 1)!$ permutations. Finally, the number of ways to pair the remaining LCD-nodes in $C_{\text{shifted right}}$ among each other is given by $(2n - 2k - 2s - 2d - 1)!!$. % Multiplication gives (\ref{cond_prob_comb})).\\
Letting $\bm{A}$ represent the event $d_{k+1}=d+1$, and using Equations~\ref{formula_for_Dk} and~\ref{cond_prob_comb}, we see that: 
\begin{align*}
\mathbb{P}\mathrm{r}\left[\bm{A}| D_k = 2k + s \right] &= \frac{(s+d)!(2n - 2k - 2s -2d - 1)!!\left(\begin{array}{c}
2n - 2k - s - d - 1 \\
s + d - 1
\end{array}\right)}{s!(2n - 2k - 2s - 1)!!\left(\begin{array}{c}
2n - 2k - s \\
s
\end{array}\right)} \\\nonumber
&=\frac{2^d(s + d)!(n - k -s)!(2n - 2k - s - d - 1)!}{(n - k - s -d)!(2n - 2k - s)!} \enspace .
\end{align*}
\end{proof}
The next lemma shows that the degree distribution for $v_k\in\mathbb{G}^{(n)}_{1}$ with index $k$ is bounded in  $M \le k \le n-M$, with $M$ representing the threshold defined as $M = \frac{n^\beta}{\log n}$ with $0 \le \beta \le 1$.

\begin{lemma}\label{lemma_2}
Let $d \le n^\alpha$ with $\alpha\in [0, \frac{1}{3}]$ and $d_{k+1}$ be the total degree of node $k + 1$, where $k\in[M, n-M]$. Then for all $(\alpha, \beta)$ such that:
\begin{equation}\label{req_1}
\begin{cases} 2\beta - \alpha > \frac{3}{2}  \\ \beta - \alpha > \frac{1}{2}\end{cases}
\enspace ,
\end{equation}  
the following is true for large $n$ :
\begin{equation}\label{degree_distr_approx}
\mathbb{P}\mathrm{r}\left[d_{k+1} = d + 1 \right] \approx (1 + o(1))\sqrt{\frac{k}{n}}\left(1 - \sqrt{\frac{k}{n}} \right)^d + \mathcal{O}(n^{-\sfrac{7}{4}}) \enspace .
\end{equation}  
\end{lemma}

\begin{proof}
Fix $D = 2k + s$ such that $|D - 2\sqrt{kn}| \le 4\sqrt{n\\log  n}$ and consider the conditional probability $\mathbb{P}\mathrm{r}\left[d_{k+1} = d + 1 | D_k = D \right]$.

\begin{claim}\label{cl_1_lemma_2}
Let $k\in[M, n- M]$ and $|D - 2\sqrt{kn}| \le 4\sqrt{n\\log  n}$. Then for large $n$:
\begin{equation}\label{approx_for_cond}
\mathbb{P}\mathrm{r}\left[d_{k+1} = d + 1 | D_k = D \right] \approx 2^d\frac{(2\sqrt{kn} - 2k + \mathcal{O}(\sqrt{n\log n}))}{(2n - 2\sqrt{kn} + \mathcal{O}(\sqrt{n\log n}))}\left(\frac{(\sqrt{n} - \sqrt{k})^2 + \mathcal{O}(\sqrt{n\log n})}{2n - 2\sqrt{kn} + \mathcal{O}(\sqrt{n\log n})}\right)^d.
\end{equation}
\end{claim}

\begin{proof}
Knowing that $D = 2k + s$, the inequality $|D - 2\sqrt{kn}| \le 4\sqrt{n\log n}$ gives the following:
\begin{align*}
&\-4\sqrt{n\log n} \le 2k + s - 2\sqrt{kn} \le 4\sqrt{n\log n} \\
&\Rightarrow 
2\sqrt{kn} - 2k - 4\sqrt{n\log n} \le s  \le 2\sqrt{kn} - 2k  + 4\sqrt{n\log n} \enspace .
\end{align*}
Using $d \le n^{\alpha}$, where $\alpha\in [0, \frac{1}{3}]$ for large, $n$ we have: 
\begin{equation}\label{appr_1}
n - k - s \approx n + k - 2\sqrt{kn} + \mathcal{O}(\sqrt{n\log n})
\end{equation}
\begin{equation}\label{appr_2}
2n - 2k - s \approx 2n - 2\sqrt{kn} + \mathcal{O}(n\log n)
\end{equation}
\begin{equation}\label{appr_3}
s + d \approx 2\sqrt{kn} - 2k + \mathcal{O}(\sqrt{n\log n}) \enspace .
\end{equation}

%where in (\ref{appr_3}) we used that $d \le n^{\alpha}$, where $\alpha\in [0, \frac{1}{3}]$ and therefore $d = \mathcal{O}(\sqrt{n\log n})$.\\ 

Using the results of Equations~\ref{appr_1},~\ref{appr_2},~and \ref{appr_3}), the conditional probability can be written as: 
\begin{align*}
\mathbb{P}\mathrm{r}\left[ d_{k+1} = d + 1| D_k = 2k + s \right] &= \frac{2^d(s + d)!(n - k - s)!(2n - 2k - s - d - 1)!}{(n - k - s - d)!(2n - 2k - s)!}  \\
&\hspace{-2em}=\frac{2^d(s + d)!(n - k - s)(n-k- s - 1)\cdots(n - k - s - d + 1)}{(2n - 2k - s - 1)(2n - 2k - s - 2)\cdots(2n - 2k - s - d)} \\
&\hspace{-2em}\approx 2^d\frac{(2\sqrt{kn} - 2k + \mathcal{O}(\sqrt{n\log n})}{(2n - 2\sqrt{kn} + \mathcal{O}(\sqrt{n\log n}))}\left(\frac{(\sqrt{n} - \sqrt{k})^2 + \mathcal{O}(\sqrt{n\log n})}{2n - 2\sqrt{kn} + \mathcal{O}(\sqrt{n\log n})}\right)^d \enspace .
\end{align*}   
\end{proof}

\begin{claim}\label{first_cond}
Let $k\in [M, n-M]$, and $d\le n^\alpha$ with $\alpha\in [0, \frac{1}{3}]$. Then for $(\alpha, \beta)$ such that $2\beta - \alpha > \frac{3}{2}$:
\begin{equation*}
\frac{(\sqrt{n} - \sqrt{k})^2}{d\sqrt{n\log n}} \to \infty
\end{equation*}
 as $n\to \infty$.
\end{claim}
\begin{proof}
Since $k\in[M, n-M]$:
\begin{align}\label{ratio_ineq}
\frac{1}{n^{1 - \beta}\log n}\le \frac{k}{n}\le 1 - \frac{1}{n^{1 - \beta}\log n}\Rightarrow
\left(1 - \sqrt{\frac{k}{n}}\right)^2 \ge \left(1 - \sqrt{1 - \frac{1}{n^{1 - \beta}\log n}}\right)^2 
\end{align}
Using a Taylor expansion of $(1 - \frac{1}{n^{1 - \beta}\log n})^{\sfrac{1}{2}} \approx 1 - \frac{1}{2n^{1 - \beta}\log n}$ for large $n$, and using Equation~\ref{ratio_ineq}, we arrive at:
\begin{align*}
\frac{(\sqrt{n} - \sqrt{k})^2}{d\sqrt{n\log n}}\ge \frac{n\left(1 - \sqrt{\frac{k}{n}}\right)^2}{n^{\sfrac{1}{2} + \alpha}\sqrt{\log n}}\ge \frac{n^{\sfrac{1}{2}+ \alpha}}{\sqrt{\log n}}\left(1 - \sqrt{1 - \frac{1}{n^{1 - \beta}\log n}}\right)^2\approx 
\frac{n^{\sfrac{1}{2} - \alpha - 2(1 - \beta)}}{4\log ^{\sfrac{5}{2}}n} \enspace .
\end{align*}
In the case that $2\beta - \alpha > \frac{3}{2}$, the last expression tends to infinity as $n\to\infty$. 
\end{proof}

\begin{claim}\label{secon_cond}
Let $k\in [M, n-M]$, and $d\le n^\alpha$ with $\alpha\in [0, \frac{1}{3}]$. Then for $(\alpha, \beta)$ such that $\beta - \alpha > \frac{1}{2}$:
\begin{equation*}
\frac{(2n - 2\sqrt{kn})}{d\sqrt{n\log n}} \to \infty
\end{equation*}
 as $n\to \infty$
\end{claim}
\begin{proof}
Using Equation~\ref{ratio_ineq} and following a similar procedure to that in Claim~\ref{first_cond} we arrive at:
\begin{align*}
\frac{(2n - 2\sqrt{kn})}{d\sqrt{n\log n}} = \frac{2\sqrt{n}\left(1 - \sqrt{\frac{k}{n}}\right)}{d\sqrt{\log n}}\ge \frac{2n^{\sfrac{1}{2} - \alpha}\left(1 - \sqrt{\frac{1}{n^{1-\beta}\log n}}\right)}{\sqrt{\log n}}\approx 
\frac{n^{-\sfrac{1}{2} - \alpha + \beta}}{\log ^{\sfrac{3}{2}}n}
\end{align*}
In the case that $\beta - \alpha > \frac{1}{2}$, the last expression tends to infinity as $n\to\infty$. 
\end{proof}

\begin{claim}\label{third_cond}
Let $k\in [M, n-M]$, and $d\le n^\alpha$ with $\alpha\in [0, \frac{1}{3}]$. Then for $\beta > 0$:
\begin{equation*}
\frac{\sqrt{n\log n}}{2\sqrt{kn} - 2k} \to 0
\end{equation*}
 as $n\to \infty$.
\end{claim}

\begin{proof}
Using Equation~\ref{ratio_ineq}:
\begin{equation}\label{ratio_ineq_2}
\frac{1}{1 - \sqrt{\frac{k}{n}}}\le\frac{1}{1 - \frac{1}{\sqrt{n^{1 - \beta}\log n}}} \enspace .
\end{equation}
Therefore,
\begin{align}\label{ratio_ineq_3}
\frac{\sqrt{n\log n}}{2\sqrt{kn} - 2k} &= \frac{\sqrt{n\log n}}{2\sqrt{k}(\sqrt{n} - \sqrt{k})} = \frac{\log ^{\sfrac{1}{2}}}{2\sqrt{k}\left(1 - \sqrt{\frac{k}{n}}\right)}  \\\nonumber
&\le \frac{\log ^{\sfrac{1}{2}}n}{2\sqrt{k}}\left(\frac{1}{1 - \frac{1}{\sqrt{n^{1-\beta}\log n}}}\right) 
\end{align}
Notice that $\frac{1}{1 - \frac{1}{\sqrt{n^{1-\beta}\log n}}} \le \frac{\sqrt{2}}{\sqrt{2} - 1}$ for large $n$. Therefore, for Equation~\ref{ratio_ineq_3}:
\begin{equation}
\frac{\sqrt{n\log n}}{2\sqrt{kn} - 2k} \le \frac{\sqrt{2}}{\sqrt{2} - 1} \frac{\log ^{\sfrac{1}{2}}n}{2\sqrt{k}} \enspace .
\end{equation}
Similarly to the proofs of the previous two claims, if $\beta > 0$ then the last expression goes to $0$ as $n\to\infty$.  
\end{proof}

Applying the results of Claims~\ref{first_cond},~\ref{secon_cond} and \ref{third_cond}  to Equation~\ref{approx_for_cond} we deduce:
\begin{align}\label{approx_for_cond_last_version}
\mathbb{P}\mathrm{r}\left[d_{k+1} = d + 1 | D_k = D \right] &\approx 2^d\frac{(2\sqrt{kn} - 2k + \mathcal{O}(\sqrt{n\log n}))}{(2n - 2\sqrt{kn} + \mathcal{O}(\sqrt{n\log n}))}\left(\frac{(\sqrt{n} - \sqrt{k})^2 + \mathcal{O}(\sqrt{n\log n})}{2n - 2\sqrt{kn} + \mathcal{O}(\sqrt{n\log n})}\right)^d \\\nonumber
&\approx (1 + o(1))\sqrt{\frac{k}{n}}\left(1 - \sqrt{\frac{k}{n}}\right)^d \enspace .
\end{align}

Finally, by letting $\bm{D}_{\alpha}= \{D:|D - 2\sqrt{kn}| \le 4\sqrt{n\log n}\}$ and using the law of total probabilities the following can be written: 
\begin{align}\label{degree_distr}
&\mathbb{P}\mathrm{r}\left[d_{k+1} = d + 1 \right] = \sum_{D}\mathbb{P}\mathrm{r}\left[d_{k+1} = d + 1,  D_k =  D\right]  \\\nonumber
&= \sum_{D\in \bm{D}_{\alpha}}\mathbb{P}\mathrm{r}\left[d_{k+1} = d + 1,  D_k =  D\right] + \sum_{D \notin \bm{D}_{\alpha}}\mathbb{P}\mathrm{r}\left[d_{k+1} = d + 1,  D_k =  D\right] 
\end{align}
At this stage, the goal is to bound each of the two sums in Equation~\ref{degree_distr}. For the first, using Equation~\ref{approx_for_cond_last_version} and Lemma~\ref{lemma_1}, for large $n$ we have:
\begin{align}\label{sum_11}
\sum_{\bm{D}_{\alpha}}\mathbb{P}\mathrm{r}\left[d_{k+1} = d + 1,  D_k =  D\right] &= \sum_{\bm{D}_{\alpha}}\mathbb{P}\mathrm{r}\left[d_{k+1} = d + 1|D_k =  D\right]\mathbb{P}\mathrm{r}\left[D_k =  D\right] \\\nonumber
&\approx(1 + o(1))\sqrt{\frac{k}{n}}\left(1 - \sqrt{\frac{k}{n}}\right)^d\sum_{\bm{D}_{\alpha}}\mathbb{P}\mathrm{r}\left[D_k =  D\right]  \\\nonumber
&=(1 + o(1))\sqrt{\frac{k}{n}}\left(1 - \sqrt{\frac{k}{n}}\right)^d\mathbb{P}\mathrm{r}\left[|D_k - 2\sqrt{kn}|\le 4\sqrt{n\log n}\right]\\ \nonumber 
&\approx(1 + o(1))\sqrt{\frac{k}{n}}\left(1 - \sqrt{\frac{k}{n}}\right)^d(1 - \mathcal{O}(n^{-\sfrac{7}{4}})) \\\nonumber
&\approx(1 + o(1))\sqrt{\frac{k}{n}}\left(1 - \sqrt{\frac{k}{n}}\right)^d \enspace .
\end{align}
Similarly, for the second sum in Equation~\ref{degree_distr}:
\begin{align}\label{sum_12}
\sum_{D \notin \bm{D}_{\alpha}}\mathbb{P}\mathrm{r}\left[d_{k+1} = d + 1,  D_k =  D\right]&\le \sum_{D\notin \bm{D}_{\alpha}}\mathbb{P}\mathrm{r}\left[D_k =  D\right] \\\nonumber
&\approx\mathbb{P}\mathrm{r}\left[|D_k - 2\sqrt{kn}| > 4\sqrt{n\log n}\right]\approx \mathcal{O}(n^{-\sfrac{7}{4}}) \enspace .
\end{align}
Combining the results of Equations~\ref{sum_11} and \ref{sum_12}:
\begin{equation*}
\mathbb{P}\mathrm{r}\left[d_{k+1} = d + 1 \right] \approx(1 + o(1))\sqrt{\frac{k}{n}}\left(1 - \sqrt{\frac{k}{n}}\right)^d + \mathcal{O}(n^{-\sfrac{7}{4}}) \enspace ,
\end{equation*}
which proves Lemma~{\ref{lemma_2}}.
\end{proof}

The next lemma computes the expectation of the number of nodes in $\mathbb{G}^{(n)}_{1}$ with a total degree of $d + 1$:
\begin{lemma}\label{lemma_3}
Let $N(d + 1)$ be the number of nodes in $\mathbb{G}^{(n)}_{1}$ that has total degree $d + 1$, where $d\le n^{\alpha}$, with $\alpha\in[0, \frac{1}{3}]$. Then for large $n$:
\begin{equation}\label{exp_approx}
\mathbb{E}(N(d + 1)) \approx \mathcal{O}(M) + \mathcal{O}(n^{-\sfrac{3}{4}}) + (1 + o(1))\frac{4n}{(d+1)(d+2)(d+3)} \enspace ,
\end{equation}
and for $(\alpha,\beta)$ such that $1 - 3\alpha > \beta$:
\begin{equation*}
\mathbb{E}(N(d + 1))\approx \frac{n}{d^3} \enspace .
\end{equation*} 
%(here $\approx*$ disregards constant factors)
\end{lemma}
\begin{proof}
Denote by $d_i$ the total degree of node $v_i$ and consider random variable  $X_i$ defined as follows:
\begin{equation*}
X_i = \begin{cases} 1 &\mbox{if } d_i = d + 1 \\ 
0 & \mbox{otherwise } \end{cases} 
\end{equation*}
Therefore, 
\begin{equation*}
\mathbb{E}(N(d + 1)) = \sum_{k=1}^{M-1}\mathbb{E}(X_i) + \sum_{k=M}^{n - M}\mathbb{E}(X_i) + \sum_{k = n-M+1}^{n}\mathbb{E}(X_i) \enspace .
\end{equation*}
Further, notice that:
\begin{equation}\label{sum_21}
\sum_{k=1}^{M-1}\mathbb{E}(X_i) + \sum_{k = n-M+1}^{n}\mathbb{E}(X_i)= \mathcal{O}(M) \enspace .
\end{equation}
Using Lemma~\ref{lemma_2}, for large $n$:
\begin{align}\label{sum_22}
\sum_{k=M}^{n - M}\mathbb{E}(X_i) &= \sum_{k=M}^{n - M}\mathbb{P}\mathrm{r}\left[d_{k+1} = d + 1 \right]\\
&\approx\mathcal{O}(n^{-\sfrac{3}{4}}) + (1 + o(1))\sum_{k = M}^{n-M}\sqrt{\frac{k}{n}}\left(1 - \sqrt{\frac{k}{n}}\right)^d \\\nonumber
&=\mathcal{O}(n^{-\sfrac{3}{4}}) + (1 + o(1))n\sum_{k = M}^{n-M}\sqrt{\frac{k}{n}}\left(1 - \sqrt{\frac{k}{n}}\right)^d\frac{1}{n}  \\\nonumber 
&= \mathcal{O}(n^{-\sfrac{3}{4}}) + (1 + o(1))\int_{0}^{1}\sqrt{u}(1 - \sqrt{u})^d du \\\nonumber
&=\mathcal{O}(n^{-\sfrac{3}{4}}) + 2(1 + o(1))\left(\frac{u^{d+1}}{d+1} - 2\frac{u^{d+2}}{d+2} + \frac{u^{d+3}}{d+3}\right)\Bigg|_{0}^{1} \\ \nonumber
&= \mathcal{O}(n^{-\sfrac{3}{4}}) + (1 + o(1))\frac{4n}{(d+1)(d+2)(d+3)} \enspace .
\end{align}
Combining Equations~\ref{sum_21} and \ref{sum_22}:
\begin{equation*}
\mathbb{E}(N(d + 1)) \approx \mathcal{O}(M) + \mathcal{O}(n^{-\sfrac{3}{4}}) + (1 + o(1))\frac{4n}{(d+1)(d+2)(d+3)} \enspace .
\end{equation*}
Notice that if $1-3\alpha > \beta$ then the term $\frac{4n}{(d+1)(d+2)(d+3)}$ dominates all other terms in Equation~\ref{exp_approx}, hence for such $(\alpha,\beta)$:
\begin{equation}
\mathbb{E}(N(d + 1))\approx \frac{n}{d^3} \enspace .
\end{equation}
\end{proof}

The next lemma establishes a similar result for graph $\mathbb{G}^{(n)}_{m}$:

\begin{lemma}
Let $N'(d + 1)$ be the number of nodes in $\mathbb{G}^{(n)}_{m}$ that has total degree $d + 1$, where $d\le n^{\alpha}$, with $\alpha\in[0, \frac{1}{3}]$. Then for large $n$:
\begin{equation*}
\mathbb{E}(N'(d + 1)) \approx \mathcal{O}(M) + \mathcal{O}(n^{-\sfrac{3}{4}}) + (1 + o(1))\frac{2m(m + 1)n}{(d+m)(d+ 1 + m)(d + m + 2)} \enspace ,
\end{equation*}
and for $(\alpha,\beta)$ such that $1 - 3\alpha > \beta$:
$\mathbb{E}(N'(d + 1))\approx \frac{n}{d^3}$.
\end{lemma}

\begin{proof}
Using the equivalence of graphs $\mathbb{G}^{(n)}_{m}$ and $\mathbb{G}^{(nm)}_{1}$, and applying the same arguments as Lemmas~\ref{lemma_1} and~\ref{lemma_2}, for node $v_{mk + j + 1}$ of graph $\mathbb{G}^{(nm)}_{1}$:
\begin{align*}
&\mathbb{P}\mathrm{r}\left[ d_{mk + j + 1} = d + 1|d_1, d_2, \ldots d_{mk + j}\right]\approx \sqrt{\frac{(mk + j)}{mn}}\left(1 - \sqrt{\frac{(mk + j)}{mn}}\right)^d \approx\sqrt{\frac{k}{n}}\left(1 - \sqrt{\frac{k}{n}}\right)^d 
\end{align*}
Therefore, for node $v'_k$ in $\mathbb{G}^{(n)}_{m}$:
\begin{align*}
\mathbb{P}\mathrm{r}\left[ d'_{k} = d + 1|d_1, d_2, \ldots d_{mk + j}\right] &= \mathcal{O}(n^{-\sfrac{7}{4}}) + (1 + o(1))\sum_{a_1 + a_2 +\cdot + a_m = d}\prod_{i=1}^{m}\sqrt{\frac{k}{n}}\left(1 - \sqrt{\frac{k}{n}}\right)^{a_i} \\\nonumber
&\approx\mathcal{O}(n^{-\sfrac{7}{4}}) + (1 + o(1))\left(\begin{array}{c}
d + m - 1 \\
m - 1
\end{array}\right)\left(\frac{k}{n}\right)^{\sfrac{m}{2}}\left(1 - \sqrt{\frac{k}{n}}\right)^{d}
\end{align*}
By applying the same technique as in Lemma~\ref{lemma_3} and using $\int_{0}^{1}\sqrt{u}^{\sfrac{m}{2}}(1 - \sqrt{u})^d du = 2\frac{(m+1)!d!}{(d + m + 2)!}$:
\begin{align}\label{exp_approx_2}
\mathbb{E}(N'(d + 1)) &\approx \mathcal{O}(M) + \mathcal{O}(n^{-\sfrac{3}{4}}) + (1+ o(1))n\left(\begin{array}{c}
d + m - 1 \\
m - 1
\end{array}\right)\int_{0}^{1}\sqrt{u}^{\sfrac{m}{2}}(1 - \sqrt{u})^d du  \\\nonumber
&\approx\mathcal{O}(M) + \mathcal{O}(n^{-\sfrac{3}{4}}) + (1+ o(1))\frac{2nm(m+1)}{(d+m)(d+m+1)(d+m+2)}
\end{align}
If $1-3\alpha > \beta$ then $\frac{2nm(m+1)}{(d+m)(d+m+1)(d+m+2)}$ dominates all other term in (\ref{exp_approx_2}), hence for such $(\alpha, \beta)$:
\begin{equation}
\mathbb{E}(N'(d + 1))\approx \frac{n}{d^3} 
\end{equation}
\end{proof}

The following lemma introduces the martingale sequence that will be the final step in the proof:
\begin{lemma}\label{lemma_4}
Let $\mathcal{F}^{t}_{m}$ be the smallest sigma field generated by random graphs $\{\mathbb{G}^{(1)}_{m}, \mathbb{G}^{(2)}_{m},\ldots, \ldots, \mathbb{G}^{(t)}_{m} \}$ and let $Y_t = \mathbb{E}(N'(d+1)|\mathcal{F}^{t}_{m})$ for $1\le t\le n$ Then,
\begin{enumerate}
\item $Y_t$ is a martingale sequence with respect to filtration $\{\mathcal{F}^{t}_{m}\}^{n}_{t=1}$
\item For $1\le t\le n$
\begin{equation}\label{mart_diff}
|Y_t - Y_{t-1}|\le 2 \enspace .
\end{equation}
\end{enumerate}
\end{lemma}
\begin{proof}
Since $Y_t$ are adapted to $\{\mathcal{F}^{t}_{m}\}^{n}_{t=1}$ and $\mathbb{E}(Y_{t+1}|{F}^{t}_{m}) = \mathbb{E}(\mathbb{E}(N'(d+1)|\mathcal{F}^{t+1}_m)|\mathcal{F}^{t}_{m}) = \mathbb{E}(N'(d+1)|\mathcal{F}^{t}_{m}) = Y_t$ it follows that $\{Y_t\}$ is martingale sequence.\\
Now notice that attachment that is made at moment $t$ does not affect the joint degree distribution of other nodes $v_i$, therefore, the change in $N'(d+1)$ is at most $2$.
\end{proof}
The next Lemma, presented without a proof, is the Azuma-Hoeffding concentration inequality for martingales

\begin{lemma}[Azuma-Hoeffding Inequality]\label{lemma_azuma}
Let $\{X_n\}$ be a martingale (or super-martingale) such that $|X_{k+1} - X_k| \le c_k$ almost surely for $k = 0,1,\ldots n-1$. Then for any $N\in \mathbb{N}$ such that $ N \le n$ and $t\in\mathbb{R}$:
\begin{equation}
\mathbb{P}\mathrm{r}\left[ |X_N - X_0| \ge t \right] \le exp\left(-\frac{t^2}{2\sum_{k = 1}^{N}c^2_k}\right)
\end{equation}

\end{lemma}

\begin{proof}
The proof can be found in~\cite{Azuma}.
\end{proof}

The next Lemma proves the Theorem~\ref{Theo:Them}:
\begin{lemma}\label{lemma_5}
For all $(\alpha, \beta)$, such that $1 - 3\alpha > \frac{1}{2}$ and
\begin{equation}\label{req_2}
\begin{cases} 2\beta - \alpha > \frac{3}{2}  \\ \beta - \alpha > \frac{1}{2} \\ 1 - 3\alpha \ge \beta \\ 0 \le \beta \le 1 \\ 0\le \alpha \le \frac{1}{3}\\
\end{cases}
\end{equation}
the following is true:
\begin{equation*}
\mathbb{P}\mathrm{r}\left[\frac{N(n,m,d)}{n}\sim\frac{1}{d^{3}}\right]\rightarrow 1, \ \text{as $n\rightarrow \infty$} \enspace .
\end{equation*} 
\end{lemma}
\begin{proof}
Notice that $Y_0 = \mathbb{E}(N(n,m,d)|{F}^{0}_{m}) = \mathbb{E}(N(n,m,d))$.  Therefore, using Lemmas~\ref{lemma_4} and \ref{lemma_azuma}:
\begin{align*}
\mathbb{P}\mathrm{r}\left[|(N(n,m,d) - \mathbb{E}(N(n,m,d))| \ge \sqrt{n\log n}\right] \le \exp\left(-\frac{n\log n}{8n}\right) &= \exp\left(-\frac{\log n}{8}\right)  \\
&= \mathcal{O}(n^{-\sfrac{1}{8}}) 
\end{align*}
Consequently,
\begin{align*}
& \mathbb{P}\mathrm{r}\left[|(N(n,m,d) - \mathbb{E}(N(n,m,d))| \le \sqrt{n\log n}\right] \approx 1 - \mathcal{O}(n^{-\sfrac{1}{8}}) 
\end{align*}  
Notice that $\mathbb{E}(N(n,m,d))\approx \frac{n}{d^3}$  for all $(\alpha, \beta)$ satisfying the system of inequalities in Equation~\ref{req_2}. For such $(\alpha, \beta)$, the following result is true:
\begin{align*}
|(N(n,m,d) - \mathbb{E}(N(n,m,d))| &\le \sqrt{n\log n}\Rightarrow \mathbb{E}(N(n,m,d)) - \sqrt{n\log n} \le N(n,m,d)  \\
& \le \mathbb{E}(N(n,m,d)) + \sqrt{n\log n}\\
&\implies \frac{n}{d^3} - \sqrt{n\log n} \le N(n,m,d) \le \frac{n}{d^3} + \sqrt{n\log n}\\\nonumber
&\Rightarrow\frac{1}{d^3} - \sqrt{\frac{\log n}{n}} \le \frac{N(n,m,d)}{n} \le \frac{1}{d^3} + \sqrt{\frac{\log n}{n}} 
\end{align*}
Please note that if $\alpha$ such that $1 - 3\alpha > \frac{1}{2}$, $\frac{1}{d^3}$ asymptotically dominates $\sqrt{\frac{\log n}{n}}$. Thus, if $(\alpha, \beta)$ such that $1 - 3\alpha >\frac{1}{2}$ and satisfying the system in Equation~\ref{req_2} for large $n$:
\begin{equation*}
\mathbb{P}\mathrm{r}\left[\frac{N(n,m,d)}{n}\sim\frac{1}{d^{3}}\right]\rightarrow 1, \ \text{as $n\rightarrow \infty$}
\end{equation*} 
\end{proof}
The last observation here is that $\alpha < \frac{1}{14}$ satisfies all conditions in Lemma~\ref{lemma_5}, thus concluding the proof of Theorem 1.
%\textbf{HERE WE NEED TO INSERT PICTURE WITH FEASIBLE ALPHA AND BETA}
\end{proof}
%\end{proof}
\subsection{Further Tightening of the Bound}
Although successful, next we introduce a novel approach capable of tightening the bound to ${n^{\sfrac{1}{6}-\epsilon}}$ for any $\epsilon > 0$. The main idea is to carefully analyze the first of vertices as these contribute most to the in-degree. The statement of Theorem~\ref{Theo:US}, demonstrating our results, is similar to that of Theorem~\ref{Theo:Them} with a major difference being the tighter bound of ${n^{\sfrac{1}{6}-\epsilon}}$. 
%The results of this analysis is provided in the following theorem: 

\begin{theorem}
\label{Theo:US}
Let $m\geq 1$ be fixed and $\left(\mathbb{G}_{m}^{(n)}\right)_{n\geq 0}$ be the process defined in Section~\ref{Sec:M2}. For any $\epsilon >0$ and $1 \leq d \leq {n^{\sfrac{1}{6}-\epsilon}}$: 
\begin{equation*}
\mathbb{P}\mathrm{r}\left[\frac{N\left(n,m,d\right)}{n}\sim \frac{1}{d^{3}}\right] \rightarrow 1, \ \text{as $n\rightarrow \infty$} \enspace ,
\end{equation*}
with $N\left(n,m,d\right)$ is the number of vertices in graph $\mathbb{G}_{m}^{(n)}$ with an indegree $d$. 
\end{theorem}
\begin{proof}
From the proof of Theorem~1, we recognize that the main restriction on inequality $\alpha$ is caused by $1-3\alpha < \beta$. This guarantees that $\frac{n}{d^3}$ dominates $M = \frac{n^{\beta}}{\log n}$ in Equation~\ref{exp_approx}. 
To relax this requirement, we study the contributions of nodes $\{v_1, v_2, \ldots, v_M \}$ and $\{v_{n-M+1},v_{n-M+2}, \ldots, v_{n}\}$ in Equation~\ref{exp_approx} more carefully. Namely:
\begin{align}\label{exp_approx_2}
\mathbb{E}(N(d + 1)) &\approx \sum_{k = 1}^{M}\mathbb{P}\mathrm{r}\left[d_{k+1} = d + 1 \right] + \frac{4n}{(d+1)(d+2)(d+3)}  \\ \nonumber
&\hspace{10em}+ \sum_{k = n - M + 1}^{n}\mathbb{P}\mathrm{r}\left[d_{k+1} = d + 1 \right] + \mathcal{O}(n^{-\sfrac{3}{4}})  \\ \nonumber
&=\mathcal{O}(\log^{2} n) + \sum_{k = \log^{2} n}^{M}\mathbb{P}\mathrm{r}\left[d_{k+1} = d + 1 \right] + \frac{4n}{(d+1)(d+2)(d+3)}  \\ \nonumber
&\hspace{10em}+ \sum_{k = n - M + 1}^{n}\mathbb{P}\mathrm{r}\left[d_{k+1} = d + 1 \right] + \mathcal{O}(n^{-\sfrac{3}{4}}) \\\nonumber
&=\sum_{k = \log^{2} n}^{M}\mathbb{P}\mathrm{r}\left[d_{k+1} = d + 1 \right] + \sum_{k = n - M + 1}^{n}\mathbb{P}\mathrm{r}\left[d_{k+1} = d + 1 \right] \\\nonumber &\hspace{10em} + \frac{4n}{(d+1)(d+2)(d+3)}  + \mathcal{O}(\log^{2}n) \enspace .
\end{align}
The proof of Theorem~\ref{Theo:US} is based on estimating the sums in Equation~\ref{exp_approx_2}:
\begin{lemma}\label{lemma_11}
Let $S_1 = \sum_{k = log^2n}^{M}\mathbb{P}\mathrm{r}\left[d_{k+1} = d + 1 \right]$, and $\alpha\in [0,\frac{1}{3}]$, $\beta
\in(0,1]$. Then:
\begin{enumerate}
\item For all $(\alpha, \beta)$ such that $\beta \le 1 - 2\alpha$, it is true for large $n$ that:
\begin{equation}
S_1 \approx M\sqrt{\frac{M}{n}} = \Theta\left(\frac{n^{\frac{3\beta}{2} - \frac{1}{2}}}{log^{\sfrac{3}{2}}n}\right)   \enspace .
\end{equation}
\item For $d < n^{\frac{1-\beta}{2}}$ and all $(\alpha, \beta)$ such that $\beta > 1 - 2\alpha$ the following is true for large $n$ :
\begin{equation}
S_1 \le \mathcal{O}\left(\frac{n^{\frac{3\beta}{2} - \frac{1}{2}}}{log^{\sfrac{3}{2}}n}\right) \enspace .
\end{equation}
\item For $d\in [n^{\frac{1-\beta}{2}}, n^{\alpha}]$ and all $(\alpha, \beta)$ such that $\beta > 1 - 2\alpha$ the following is true for large $n$ :
\begin{equation}
S_1 \approx \mathcal{O}\left(\frac{n}{d^3}\right) \enspace .
\end{equation}
\end{enumerate}
\end{lemma}
\begin{proof}
The main idea is to show that when $k\in [log^2n, M]$, the expressions in Equation~\ref{degree_distr_approx} still hold.

\begin{claim}\label{prob_claim}
Let $d \le n^\alpha$ with $\alpha\in [0, \frac{1}{3}]$ and $\beta\in (0,1]$, and let $d_{k+1}$ be the total degree of node $k + 1$, where $k\in[log^2n, M]$. Then for large $n$:
\begin{equation}\label{degree_distr_approx_2}
\mathbb{P}\mathrm{r}\left[d_{k+1} = d + 1 \right] \approx (1 + o(1))\sqrt{\frac{k}{n}}\left(1 - \sqrt{\frac{k}{n}} \right)^d + \mathcal{O}(n^{-\sfrac{7}{4}})  \enspace .
\end{equation}
\end{claim}
\begin{proof}
Following the proof of Lemma~\ref{lemma_2}, the following facts need to be verified: 
\begin{enumerate}
\item Let $k\in [\log^{2}n, M]$ and $|D - 2\sqrt{kn}| \le 4\sqrt{n\log n}$. Then for large $n$:
\begin{align}\label{approx_for_cond_2}
\mathbb{P}\mathrm{r}\left[d_{k+1} = d + 1 | D_k = D \right] &\approx 2^d\frac{(2\sqrt{kn} - 2k + \mathcal{O}(\sqrt{n\log n}))}{(2n - 2\sqrt{kn} + \mathcal{O}(\sqrt{n\log n}))} \\\nonumber
&\hspace{5em}\times \left(\frac{(\sqrt{n} - \sqrt{k})^2 + \mathcal{O}(n\log n)}{2n - 2\sqrt{kn} + \mathcal{O}(\sqrt{n\log n})}\right)^d.
\end{align}
\item Let $k\in [\log^{2}n, M]$, and $d\le n^\alpha$ with $\alpha\in [0, \frac{1}{3}]$ and $\beta\in [0,1]$, then:
\begin{equation*}
\frac{(\sqrt{n} - \sqrt{k})^2}{d\sqrt{nlogn}} \to \infty
\end{equation*}
 as $n\to \infty$.
\item Let $k\in [\log^{2}n, M]$, and $d\le n^\alpha$ with $\alpha\in [0, \frac{1}{3}]$ and $\beta\in [0,1]$, then:
\begin{equation*}
\frac{(2n - 2\sqrt{kn})}{d\sqrt{n\log n}} \to \infty
\end{equation*}
 as $n\to \infty$.
\item Let $k\in [\log^{2}n, M]$, and $d\le n^\alpha$ with $\alpha\in [0, \frac{1}{3}]$. Then for $\beta\in (0,1]$:
\begin{equation*}
\frac{\sqrt{n\log n}}{2\sqrt{kn} - 2k} \to 0
\end{equation*}
 as $n\to \infty$.
\end{enumerate}
The above facts can be proven accordingly: 
\begin{enumerate}
\item Similarly to Claim~\ref{cl_1_lemma_2} where $|D_k - 2\sqrt{kn}| \le 4\sqrt{n\log n}$, we can conclude that: 
\begin{equation*}
2\sqrt{kn} - 2k - 4\sqrt{n\log n} \le s  \le 2\sqrt{kn} - 2k  + 4\sqrt{n\log n} \enspace .
\end{equation*}
Therefore, the approximations of Equations~\ref{appr_1},~\ref{appr_2}, and~\ref{appr_3} hold.
Since $d\le n^{\alpha}$ with $\alpha\in[0,\frac{1}{3}]$, then $s + d\approx 2\sqrt{kn} - 2k + \mathcal{O}(\sqrt{n\log n})$. Hence, the expression in Equation~\ref{cond_prob_expression} can be written as:
\begin{align*}
\mathbb{P}\mathrm{r}&\left[ d_{k+1} = d + 1| D_k = 2k + s \right]  = \frac{2^d(s + d)!(n - k - s)!(2n - 2k - s - d - 1)!}{(n - k - s - d)!(2n - 2k - s)!}  \\\nonumber
&\hspace{6em}=\frac{2^d(s + d)!(n - k - s)(n-k- s - 1)\cdots(n - k - s - d + 1)}{(2n - 2k - s - 1)(2n - 2k - s - 2)\cdots(2n - 2k - s - d)}  \\\nonumber
&\hspace{6em}\approx 2^d\frac{(2\sqrt{kn} - 2k + \mathcal{O}(\sqrt{n\log n})}{(2n - 2\sqrt{kn} + \mathcal{O}(\sqrt{n\log n}))}\left(\frac{(\sqrt{n} - \sqrt{k})^2 + \mathcal{O}(\sqrt{n\log n})}{2n - 2\sqrt{kn} + \mathcal{O}(\sqrt{n\log n})}\right)^d \enspace .
\end{align*}
Thus, Fact 1 is established. 

\item For $k\in [\log^{2}n, M]$, we have:
\begin{align}\label{inter_res_2}
&\log n\le \sqrt{k}\le \sqrt{M} \Rightarrow \\\nonumber
&\sqrt{n} - \sqrt{M} \le \sqrt{n} - \sqrt{k}\le \sqrt{n} - \log n \enspace .
\end{align}
Further, since $d\le n^{\alpha}$:
\begin{align*}
\frac{(\sqrt{n} - \sqrt{k})^2}{d\sqrt{n\log n}} \ge \frac{(\sqrt{n} - \sqrt{M})^2}{d\sqrt{n\log n}} & = \frac{n\left(1 - \sqrt{\frac{M}{n}}\right)^2}{d\sqrt{n\log n}} \\
&\ge \frac{n^{1 - \sfrac{1}{2} - \alpha}\left(1 - \sqrt{\frac{M}{n}}\right)^2}{\sqrt{\log n}} \approx 
\frac{n^{1 - \sfrac{1}{2} - \alpha}}{\sqrt{\log n}} \enspace .
\end{align*}
If $\alpha < \frac{1}{2}$ then the last expression tends to $\infty$ as $n\to\infty$, thus establishing Fact 2. 

\item Using Equation~\ref{inter_res_2}, we recognize:
\begin{align*}
\frac{2n - 2\sqrt{kn}}{d\sqrt{n\log n}} &= \frac{2(\sqrt{n} - \sqrt{k})}{d\sqrt{\log n}} \ge \frac{2(\sqrt{n} - \sqrt{M})}{d\sqrt{\log n}} \\
&= \frac{2\sqrt{n}\left(1 - \sqrt{\frac{M}{n}}\right)}{d\sqrt{\log n}} \ge 
\frac{2n^{\sfrac{1}{2} - \alpha}\left(1 - \sqrt{\frac{M}{n}}\right)}{\sqrt{\log n}} \approx \frac{n^{\sfrac{1}{2} - \alpha}}{\sqrt{\log n}}
\enspace .
\end{align*}
Fact 3 can be established if $\alpha < \frac{1}{2}$. %then the last expression goes to infinity as $n\to\infty$.  Fact 3 is established.

\item Using Equation~\ref{inter_res_2}, we recognize:
\begin{equation*}
\frac{1}{\sqrt{n} - \log n}\le \frac{1}{\sqrt{n} - \sqrt{k}}\le \frac{1}{\sqrt{n} - \sqrt{M}} \enspace .
\end{equation*}
Therefore,
\begin{align*}
&\frac{\sqrt{n\log n}}{2\sqrt{kn} - 2k} \le \frac{\sqrt{n\log n}}{\log n(\sqrt{n} - \sqrt{M})} = \frac{1}{\sqrt{\log n}\left(1 - \sqrt{\frac{M}{n}}\right)} \approx \frac{1}{\sqrt{\log n}} \enspace .
\end{align*}
This expression tends to infinity as $n\to\infty$, thus establishing Fact 4. 
\end{enumerate}
Having the above facts, we can write: 
\begin{align*}
\mathbb{P}\mathrm{r}\left[d_{k+1} = d + 1 | D_k = D \right] &\approx 2^d\frac{(2\sqrt{kn} - 2k + \mathcal{O}(\sqrt{nlogn}))}{(2n - 2\sqrt{kn} + \mathcal{O}(\sqrt{nlogn}))}\left(\frac{(\sqrt{n} - \sqrt{k})^2 + \mathcal{O}(nlogn)}{2n - 2\sqrt{kn} + \mathcal{O}(\sqrt{nlogn})}\right)^d  \\\nonumber
&\approx(1 + o(1))\sqrt{\frac{k}{n}}\left(1 - \sqrt{\frac{k}{n}}\right)^d 
\end{align*}

As in Lemma~\ref{lemma_2}, we have that for all $\alpha\in [0,\frac{1}{3}]$, $\beta\in(0,1]$, and for large $n$:
\begin{equation}\label{degree_distr_approx_2}
\mathbb{P}\mathrm{r}\left[d_{k+1} = d + 1 \right] \approx (1 + o(1))\sqrt{\frac{k}{n}}\left(1 - \sqrt{\frac{k}{n}} \right)^d + \mathcal{O}(n^{-\sfrac{7}{4}}) \enspace .
\end{equation}
Therefore, $S_1$ can be written as:
\begin{align}
S_1 =  \sum_{k = log^2n}^{M}\mathbb{P}\mathrm{r}\left[d_{k+1} = d + 1 \right]  \approx (1 + o(1))\sum_{k = log^2n}^{M}\sqrt{\frac{k}{n}}\left(1 - \sqrt{\frac{k}{n}} \right)^d + \mathcal{O}(n^{-\sfrac{3}{4}}) \enspace .
\end{align}

\begin{claim}\label{case_claim}
Let $M = \frac{n^{\beta}}{logn}$ and $d\le n^{\alpha}$ with $\alpha\in[0,\frac{1}{3}]$ and $\beta\in(0,1]$. Denote $\mathcal{Z}(n,\alpha, \beta)= \sum_{k = \log^{2}n}^{M}\sqrt{\frac{k}{n}}\left(1 - \sqrt{\frac{k}{n}} \right)^d$.  Then:
\begin{enumerate}
\item For all $(\alpha, \beta)$ such that $\beta \le 1 - 2\alpha$:
\begin{equation}\label{case_1}
\mathcal{Z}(n,\alpha, \beta) = \Theta\left(\frac{n^{\frac{3\beta}{2} - \frac{1}{2}}}{\log^{\sfrac{3}{2}}n}\right)
\end{equation}
 for large $n$.
\item For all $(\alpha, \beta)$ such that $\beta > 1 - 2\alpha$ and $d< n^{\frac{1-\beta}{2}}$
\begin{equation}\label{case_2}
\mathcal{Z}(n,\alpha, \beta) \le \Theta\left(\frac{n^{\frac{3\beta}{2} - \frac{1}{2}}}{\log^{\sfrac{3}{2}}n}\right)
\end{equation}
 for large $n$.
 \item For all $(\alpha, \beta)$ such that $\beta > 1 - 2\alpha$ and $d\in [n^{\frac{1-\beta}{2}}, n^{\alpha}]$
\begin{equation}\label{case_3}
\mathcal{Z}(n,\alpha, \beta) \le \mathcal{O}\left(\frac{n}{d^3}\right)
\end{equation}
 for large $n$.
\end{enumerate}
\end{claim}

\begin{proof}We separate the proof in two parts, and consider each case separately:
\begin{enumerate}
\item \textbf{Case $\beta \le 1 - 2\alpha$:} Using the fact that $\sum_{l=1}^{K}\sqrt{l}\approx K\sqrt{K}$, the straightforward upper bound for $\mathcal{Z}(n,\alpha, \beta)$ has the following form:
\begin{align}\label{uppb_1}
&Z(n,\alpha, \beta) = \sum_{k = log^2n}^{M}\sqrt{\frac{k}{n}}\left(1 - \sqrt{\frac{k}{n}} \right)^d \le \sum_{k = log^2n}^{M}\sqrt{\frac{k}{n}} \approx M\sqrt{\frac{M}{n}} = \Theta\left(\frac{n^{\frac{3\beta}{2} - \frac{1}{2}}}{log^{\sfrac{3}{2}}n}\right) \enspace .
\end{align}
To arrive at the lower bound, we make use of Bernoulli's inequality to get:
\begin{align*}
\mathcal{Z}(n, \alpha, \beta) = \sum_{k = \log^{2}n}^{M}\sqrt{\frac{k}{n}}\left(1 - \sqrt{\frac{k}{n}} \right)^d &\ge \sum_{k = \log^{2}n}^{M}\sqrt{\frac{k}{n}}\left(1 - d\sqrt{\frac{k}{n}} \right)  \\ 
&\hspace{-3em}= \sum_{k = \log^{2}n}^{M}\sqrt{\frac{k}{n}} - \frac{d}{n}\sum_{k = \log^{2}n}^{M}k  \approx M\sqrt{\frac{M}{n}} - \frac{dM^2}{n} \\
&\approx \frac{n^{\frac{3\beta}{2} -\frac{1}{2}}}{\log^{\sfrac{3}{2}}n} - \frac{n^{\alpha + 2\beta - 1}}{\log^{2}n} \approx \frac{n^{\frac{3\beta}{2} -\frac{1}{2}}}{\log^{\sfrac{3}{2}}n} 
\enspace ,
\end{align*} 
where the last step is performed since $\beta \le 1 - 2\alpha$ implies $\frac{3\beta}{2} - \frac{1}{2}\ge \alpha + 2\beta - 1$. Hence:
\begin{equation*}
\mathcal{Z}(n,\alpha,\beta) = \Theta\left(\frac{n^{\frac{3\beta}{2} - \frac{1}{2}}}{log^{\sfrac{3}{2}}n}\right) \enspace ,
\end{equation*}
concluding the proof of Equation~\ref{case_1}.
\item \textbf{Case $\beta > 1 - 2\alpha$}. First denote $c_k = \sqrt{\frac{k}{n}}\left(1 - \sqrt{\frac{k}{n}}\right)^d$, and consider the ratio:
\begin{align*}
r_k = \frac{c_k}{c_{k+1}} =  \sqrt{\frac{k}{k+1}}\left(\frac{1 - \sqrt{\frac{k}{n}}}{1 - \sqrt{\frac{k+1}{n}}}\right)^d &= \sqrt{\frac{k}{k+1}}\left(1 + \frac{\sqrt{k+1} - \sqrt{k}}{\sqrt{n} - \sqrt{k+1}}\right)^d  \\
&\hspace{-1em}=\sqrt{\frac{k}{k+1}}\left(1 + \frac{1}{(\sqrt{k+1} + \sqrt{k})(\sqrt{n} - \sqrt{k+1})}\right)^d  \\
&= \sqrt{\frac{k}{k+1}}e^{d\log\left(1 + \frac{1}{(\sqrt{k+1} + \sqrt{k})(\sqrt{n} - \sqrt{k+1})}\right)} \\\nonumber
&\approx\sqrt{\frac{k}{k+1}}e^{\frac{d}{(\sqrt{k+1} + \sqrt{k})(\sqrt{n} - \sqrt{k+1})}}  \\
&= \sqrt{1 - \frac{1}{k+1}}e^{\frac{d}{\sqrt{n(k+1)}\left(1 + \sqrt{\frac{k}{k+1}}\right)\left(1 - \sqrt{\frac{k+1}{n}}\right)}} \\\nonumber
&\approx\left(1 - \frac{1}{2(k+1)}\right)e^{\frac{d\left(1 + \sqrt{\frac{k+1}{n}}\right)}{2\sqrt{k(n+1)}}}  \\
&= \left(1 - \frac{1}{2(k+1)}\right)e^{\frac{d}{2\sqrt{n(k+1)}} + \frac{d}{2n}} \\
&\approx \left(1 - \frac{1}{2k}\right)\left(1 + \frac{d}{2\sqrt{kn}}\right)\left(1 + \frac{d}{2n}\right) \\\nonumber
&\approx 1 - \frac{1}{2k} + \frac{d}{2\sqrt{kn}} + \frac{d}{2n} \approx 1 - \frac{1}{2k} + \frac{d}{2\sqrt{kn}}
\enspace ,
\end{align*}
where the last step is followed since: $\frac{d}{2n} = o\left(\frac{d}{2\sqrt{kn}}\right)$. Therefore, if $k \ge \frac{n}{d^2}$, then $c_k$ decreases as $k$ increases. Further, if $k < \frac{n}{d^2}$, then $c_k$ increases with $k$. Here, we need to consider the following two cases: 
\begin{enumerate}
\item \textbf{Case: $d\le n^{\frac{1-\beta}{2}}$:} Notice that in this case $\frac{n}{d^2}\ge M$ and $c_k$ increases with $k$, and therefore:
\begin{align*}
\mathcal{Z}(n,\alpha, \beta) = \sum_{k = \log^{2}n}^{M}\sqrt{\frac{k}{n}}\left(1 - \sqrt{\frac{k}{n}} \right)^d &\le M\sqrt{\frac{M}{n}}\left(1 - \sqrt{\frac{M}{n}}\right)^d  \\
&\le M\sqrt{\frac{M}{n}} = 
\Theta\left(\frac{n^{\frac{3\beta}{2} - \frac{1}{2}}}{\log^{\sfrac{3}{2}}n}\right) \enspace ,
\end{align*}
which finishes the proof of Equation~\ref{case_2}.
 
\item \textbf{Case: $d\in [n^{\frac{1-\beta}{2}}, n^{\alpha}]$:}\\ Notice that in this case $\log^{2}n < \sfrac{n}{d^2} < M$ and $c_k$ increases when $k\le \frac{n}{d^2}$ and decreases when $\frac{n}{d^2}\le k\le M$. Denote $m_0 = \sfrac{n}{d^2}$ and split the sum $\mathcal{Z}(n,\alpha,\beta)$:
\begin{align}\label{sum_z}
\mathcal{Z}(n,\alpha, \beta) = \underbrace{\sum_{k = \log^{2}n}^{m_0}\sqrt{\frac{k}{n}}\left(1 - \sqrt{\frac{k}{n}} \right)^d}_{\mathcal{Z}_1} + \underbrace{\sum_{k = m_0}^{M}\sqrt{\frac{k}{n}}\left(1 - \sqrt{\frac{k}{n}} \right)^d}_{\mathcal{Z}_2}
\enspace .
\end{align}
Now, consider each of $\mathcal{Z}_1$ and $\mathcal{Z}_2$ separately:
\begin{align}\label{sum_z1}
\mathcal{Z}_1 &= \sum_{k = log^2n}^{m_0}\sqrt{\frac{k}{n}}\left(1 - \sqrt{\frac{k}{n}} \right)^d \le m_0c_{m_0} = m_0\sqrt{\frac{m_0}{n}}\left(1 - \sqrt{\frac{m_0}{n}}\right)^d \\\nonumber
&=m_0\sqrt{\frac{m_0}{n}}\left(1 - \frac{1}{d}\right)^d \approx m_0\sqrt{\frac{m_0}{n}} = \mathcal{O}\left(\frac{n}{d^3}\right)
\enspace ,
\end{align}
where we made of use of $(1 - \frac{1}{x})^x \le e^{-1} $ for $x\ge 1$.  
\begin{align}\label{sum_z2}
\mathcal{Z}_2 &= \sum_{k = m_0}^{M}\sqrt{\frac{k}{n}}\left(1 - \sqrt{\frac{k}{n}} \right)^d \le \sum_{k = m_0}^{M}\sqrt{\frac{k}{n}}e^{\left(-d\sqrt{\frac{k}{n}}\right)} \le \int_{m_0}^{\infty}\sqrt{\frac{x}{n}}e^{\left(-d\sqrt{\frac{x}{n}}\right)} dx \\\nonumber
&=2n\int_{\frac{1}{d}}^{\infty}t^2e^{-dt} dt = 2n\left(-\frac{t^2}{d}d^{-dt} - \frac{2t}{d^2}e^{-dt} - \frac{2}{d^3}e^{-dt}\right)\Bigg|_{\frac{1}{d}}^{\infty} \\
& = \frac{10n}{d^3}e^{-1} = \mathcal{O}\left(\frac{n}{d^3}\right)
\enspace .
\end{align}
Using Equations~\ref{sum_z1} and ~\ref{sum_z2} in Equation~\ref{sum_z} gives the result in Equation~\ref{case_3}.
\end{enumerate}
\end{enumerate}
\end{proof}
Lemma~\ref{lemma_11} follows by applying Claim~\ref{case_claim} to Claim~\ref{prob_claim}. 
\end{proof}

\begin{lemma}\label{lemma_12}
Let $S_2 = \sum_{k = n - M + 1}^{n}\mathbb{P}\mathrm{r}\left[d_{k+1} = d + 1 \right]$, and $\alpha\in [0,\frac{1}{3}]$, $\beta
\in(0,1)$. Then for large $n$
\begin{equation}\label{sec_sum}
S_2 = \mathcal{O}\left(\frac{n}{d^3}\right) \enspace .
\end{equation}
\end{lemma}

\begin{proof}
In this case, $k \in[n - M + 1, n]$. Notice that if $d > M$, then $S_2 = 0$, because it is impossible for node $v_k$ to accumulate $d$ links from later nodes. Here, $\alpha < \beta$. Considering $v_{n - M + 1}\in \mathbb{G}^{(n)}_{m}$.
\begin{claim}
Let $d_{n- M + 1}$ be the degree of node $v_{n - M + 1}\in\mathbb{G}^{(n)}_{m}$, and $d\le n^{\alpha}$ with $\alpha < \beta$ Then:
\begin{equation}
\mathbb{P}\mathrm{r}\left[d_{n - M + 1} = d + 1\right] \approx  \frac{M^d}{2^dn^d} \enspace .
\end{equation}
\end{claim}

\begin{proof}
Node $v_{n-M+1}$ has total degree $d+1$ iff exactly $d$ nodes from $\{v_{n-M+2}, v_{n-M+3}, \ldots, v_{n}\}$ connect to $v_{n-M+1}$. Therefore, the total number of ways to pick $d$ nodes out from $M - 1$ is given by $\left(\begin{array}{c}M - 1 \\ d\end{array}\right)$. 

Let $\mathbb{P}\mathrm{r}\left[d_{n - M + 1} = d + 1, \{v_{n - M + r_1}, v_{n - M + r_2}, \ldots, v_{n - M + r_d}\}\right]$ represent the probability that nodes $\{v_{n-M + r_1}, v_{n-M + r_2}, \ldots, v_{n - M + r_d}\}$ make connection with $v_{n - M + 1}$, where $r_i\in [2, M]$. Notice, that because $r_1 \ge 2, r_2\ge 3,\ldots r_d\ge d+1$:
\begin{align*}
\mathbb{P}\mathrm{r}\left[d_{n - M + 1} = d + 1, \{v_{r_1}, v_{r_2}, \ldots, v_{r_d}\}\right] &= \frac{1}{2(n-M+ r_1 -1)} \cdots \frac{d}{2(n-M+r_d -1)}  \\\nonumber
&\le\frac{d!}{2^d(n-M+ 1)(n-M+ 2)\cdots(n-M+ d)} \enspace .
\end{align*}
Therefore,
\begin{align*}
\mathbb{P}\mathrm{r}\left[d_{n - M + 1} = d + 1\right] &\le \left(\begin{array}{c}M - 1 \\ d\end{array}\right)\frac{d!}{2^d(n-M+1)(n-M+2)\cdots n}  \\\nonumber
&=\frac{1}{2^d}\frac{(M - d)(M - d + 1)\cdots (M-1)}{(n - M + 1)(n - M + 2)\cdots(n - M + d)}  \\
&= \frac{1}{2^d}\frac{M^d}{n^d}\frac{(1 - \frac{d}{M})\cdots(1 - \frac{1}{M})}{(1 - \frac{M-1}{n})\cdots (1 - \frac{M-d}{n})} \\\nonumber
&\approx\frac{M^d}{2^dn^d} \enspace .
\end{align*}		
\end{proof}  
It is easy to see that $\mathbb{P}\mathrm{r}\left[d_{k} = d + 1\right] \le \mathbb{P}\mathrm{r}\left[d_{n - M + 1} = d + 1\right]$ for other nodes $v_k$ with $k\in [n-  M + 1, n]$.
Therefore, using that $2^{d+1} \ge d^3$ and $\beta(d+1) < d$ for large $n$,
\begin{align*}
&S_2 = \sum_{k = n - M + 1}^{n}\mathbb{P}\mathrm{r}\left[d_{k+1} = d + 1 \right] \le  M\frac{M^d}{2^dn^d} = \frac{M^{d + 1}}{2^dn^d} \le \frac{n^{\beta(d+1)}}{2^dn^{d}} = \frac{n^{\beta(d+1) - d}}{2^d} \le \frac{2n}{d^3}
\enspace .
\end{align*}
\end{proof}
Lemmas~\ref{lemma_11} and~\ref{lemma_12} applied in Equation~\ref{exp_approx_2} give the following three inequalities:
\begin{enumerate}
\item For all $\alpha\in [0,\frac{1}{3}]$, $\beta\in (0,1)$ such that 
\begin{equation}\label{system_1}
\begin{cases}
2\beta - \alpha > \frac{3}{2} \\ \beta - \alpha > \frac{1}{2}\\\beta \le 1 - 2\alpha \\\frac{3\beta}{2} - \frac{1}{2} \le 1 - 3\alpha
\end{cases}
\end{equation}
for large $n$
\begin{equation*}
\mathbb{E}(N(d + 1)) \approx \frac{n}{d^3} \enspace .
\end{equation*}
\item For $d < n^{\frac{1 -\beta}{2}}$ and for all $\alpha\in [0,\frac{1}{3}]$, $\beta\in (0,1)$ such that 
\begin{equation}\label{system_2}
\begin{cases}
2\beta - \alpha > \frac{3}{2} \\ \beta - \alpha > \frac{1}{2}\\\beta > 1 - 2\alpha \\ \frac{3\beta}{2} - \frac{1}{2} \le 1 - 3\alpha
\end{cases}
\end{equation}
for large $n$
\begin{equation*}
\mathbb{E}(N(d + 1)) \approx \frac{n}{d^3} \enspace .
\end{equation*}
\item For $d\in[n^{\frac{1 -\beta}{2}}, n^{\alpha}]$ and for all $\alpha\in [0,\frac{1}{3}]$, $\beta\in (0,1)$ such that 
\begin{equation}\label{system_3}
\begin{cases}
2\beta - \alpha > \frac{3}{2} \\ \beta - \alpha > \frac{1}{2}\\\beta > 1 - 2\alpha
\end{cases}
\end{equation}
for large $n$
\begin{equation*}
\mathbb{E}(N(d + 1)) \approx \frac{n}{d^3} \enspace .
\end{equation*}
\end{enumerate}
The application of the Azuma-Hoeffding inequality requires an additional constraint of: $1 - 3\alpha <\frac{1}{2}$. Taking this into account, it is easy to see that for any small enough $\epsilon > 0$,  $\alpha = \frac{1}{6} - \epsilon$ and $\beta = \frac{7}{8}+ 2\epsilon$ are feasible. Therefore, for $d\in [n^{\frac{1}{16}-\epsilon},  n^{\frac{1}{6} -\epsilon}]$:
\begin{equation*}
\mathbb{P}\mathrm{r}\left[\frac{N(n,m,d)}{n}\sim\frac{1}{d^{3}}\right]\rightarrow 1, \ \text{as $n\rightarrow \infty$} \enspace .
\end{equation*}
Using this result with the result of Theorem~\ref{Theo:Them} establishes the statement of Theorem~\ref{Theo:US}.
\end{proof}
\end{proof}
\section{Discussion \& Conclusion} \label{Sec:Conclusion}
This paper presents two new bounds on the degree distribution of networks following from the preferential attachment model.  Due to its generality and impact, we adopted the preferential attachment formalization proposed by Bollob\'{a}s \etal\cite{Bollobas1} as the framework for our analysis. These new bounds were presented in two Theorems. Theorem~\ref{Theo:Them} shows that we are able to tighten the bound of Bollob\'{a}s \etal\cite{Bollobas1} from $d \leq n^{\sfrac{1}{15}}$ to $d \leq n^{\sfrac{1}{14}-\epsilon}$ for any $\epsilon > 0$. Theorem~\ref{Theo:US} then shows that we can further improve this bound to $d \leq n^{\sfrac{1}{6} -\epsilon}$ for any $\epsilon > 0$, yielding the tightest bound currently available on the degree distribution. We achieve this bound by introducing a novel technique capable of carefully analyzing the first set of vertices in complex networks. 

An interesting question is to what extent can the bound on Theorem~\ref{Theo:US} be further tightened? Definitely, this is an interesting direction for future work. Here, however, we provide a corollary showing that for $\sfrac{1}{6} \leq  \bm{\alpha} < \sfrac{1}{3}$ (the valid $\bm{\alpha}$ range to be considered), the probability that the portion of vertices exhibiting a power-law distribution tends to zero as $n \rightarrow \infty$: 

\begin{corollary}
\label{Theo:US2}
Let $m\geq 1$ be fixed and $\left(\mathbb{G}_{m}^{(n)}\right)_{n\geq 0}$ be the process defined in Section~\ref{Sec:M2}. For ${n^{\sfrac{1}{6}}}\leq  d < {n^{\sfrac{1}{3}}}$: 
\begin{equation*}
\mathbb{P}\mathrm{r}\left[\frac{N\left(n,m,d\right)}{n} = 0 \right] \rightarrow 1, \ \text{as $n\rightarrow \infty$}
\end{equation*}
with $N\left(n,m,d\right)$ is the number of vertices in graph $\mathbb{G}_{m}^{(n)}$ with an indegree $d$. %
\end{corollary}
This corollary establishes the fact that for $d$-ranges greater than $\sfrac{1}{6}$, the probability of attaining a fraction of the vertices following a power-law degree distribution tends to zero as $n$ grows large. Therefore, from such a perspective, the bound in Theorem~\ref{Theo:US} can be considered tight. 
\begin{proof}
The proof is quite straight-forward. It is enough to recognize that if $\bm{\alpha} > \sfrac{1}{6}$, the term $\sqrt{\frac{\log n}{n}}$ as opposed to $\frac{n}{d^{3}}$ will dominate the expectation $\mathbb{E}(N(n,m,d))$. Therefore: 
\begin{align*}
0\leq &\lim_{n\rightarrow \infty}\frac{N(n,m,d)}{n} \leq \lim_{n\rightarrow \infty} \sqrt{\frac{\log n}{n}} =0 \enspace .
%0 \leq &\lim_{n\rightarrow \infty}\frac{N(n,m,d)}{n} \leq 0
\end{align*}
This implies that the probability to find a fraction of the vertices following a power-law degree distribution tends to zero as $n$ grows large. 
\end{proof}

%There several interesting future work directions building on this work. First, we would like to provide an empirical validation for the above theorems on real-world data. Precisely, we are interested in determining whether the degree distribution for the case presented in Corollary~\ref{Theo:US2} follow a separate law.  
%Furthermore, there is an opportunity to study the effect of the attained bounds on the approaches adopting and extending the general formalization of Bollob\'{a}s in future work.

\end{document}